\documentclass[12pt]{amsart}
\usepackage[colorlinks=true]{hyperref}
\usepackage[all]{xy}
\usepackage{amsmath,amssymb,amsthm}
\usepackage{txfonts}
\usepackage{enumerate}
\usepackage{graphicx}
\theoremstyle{plain}
	\newtheorem{thm}{Theorem}[section]
	\newtheorem{prp}[thm]{Proposition}
	\newtheorem{lem}[thm]{Lemma}
	\newtheorem{cor}[thm]{Corollary}
\theoremstyle{definition}
	\newtheorem{dfn}[thm]{Definition}
	\newtheorem{ex}[thm]{Example}
	\newtheorem{conjecture}[thm]{Conjecture}
\theoremstyle{remark}
	\newtheorem{rem}[thm]{Remark}

\newcommand{ \Map}{\operatorname{Map}}
\newcommand{ \cat}{\operatorname{cat}}

\newcommand{ \ad}{\operatorname{ad}}
\newcommand{ \GL}{\operatorname{GL}}

\newcommand{ \SU}{\operatorname{SU}}

\begin{document}
\title{Mapping spaces from projective spaces}
\author{Mitsunobu Tsutaya}
\address{Department of Mathematics, Kyoto University, Kyoto 606-8502, Japan}
\email{tsutaya@math.kyoto-u.ac.jp}
\begin{abstract}
We denote the $n$-th projective space of a topological monoid $G$ by $B_nG$ and the classifying space by $BG$.
Let $G$ be a well-pointed topological monoid having the homotopy type of a CW complex and $G'$ a well-pointed grouplike topological monoid.
We prove that there is a natural weak equivalence between the pointed mapping space $\mathrm{Map}_0(B_nG,BG')$ and the space $\mathcal{A}_n(G,G')$ of all $A_n$-maps from $G$ to $G'$.
Moreover, if we suppose $G=G'$, then an appropriate union of path-components of $\mathrm{Map}_0(B_nG,BG)$ is delooped.

This fact has several applications.
As the first application, we show that the evaluation fiber sequence $\mathrm{Map}_0(B_nG,BG)\rightarrow\mathrm{Map}(B_nG,BG)\rightarrow BG$ extends to the right.
As other applications, we investigate higher homotopy commutativity, $A_n$-types of gauge groups, $T_k^f$-spaces and homotopy pullback of $A_n$-maps.
The concepts of $T_k^f$-space and $C_k^f$-space were introduced by Iwase--Mimura--Oda--Yoon, which is a generalization of $T_k$-spaces by Aguad\'e.
In particular, we show that the $T_k^f$-space and the $C_k^f$-space are exactly the same concept and give some new examples of $T_k^f$-spaces.
\end{abstract}
\date{}
\maketitle

\section{Introduction}
\label{section_Introduction}

In this paper, we study maps between topological monoids which preserve associativity up to higher homotopy.
In homotopy theory, homomorphisms are sometimes too restrictive.
M. Sugawara \cite{Sug60} studied the condition for a map $f\colon G\rightarrow G'$ between topological monoids to be the loop of a map $BG\rightarrow BG'$ between the classifying spaces.
Roughly, his answer is that $f$ is a loop map if $f$ preserves the multiplications on $G$ and $G'$ up to infinitely higher homotopy associativity.
In the proof of it, he used the Dold--Lashof construction \cite{DL59}.
After that, J. Stasheff \cite{Sta63b} introduced \textit{$A_n$-maps}, which are maps between topological monoids  preserving $n$-th homotopy associativity.
As a generalization of Sugawara's result, he gave an equivalent condition for a map being an $A_n$-map using the finite stages of the Dold--Lashof construction.
The $n$-th Dold--Lashof construction of a topological monoid $G$ is called the \textit{$n$-th projective space} $B_nG$ since those of $\mathbb{Z}/2\mathbb{Z}$, $S^1$ and $S^3$ are the classical projective spaces $\mathbb{R}P^n$, $\mathbb{C}P^n$ and $\mathbb{H}P^n$, respectively.

We will refine Stasheff's result using mapping spaces.
More precisely, our main result Theorem \ref{mainthm} is the weak equivalence
\begin{align*}
\mathcal{A}_n(G,G')\simeq\Map_0(B_nG,BG'),
\end{align*}
where $\mathcal{A}_n(G,G')$ is the space of $A_n$-maps with $A_n$-forms between $G$ and $G'$, and $\Map_0(B_nG,BG')$ is the space of pointed maps between $B_nG$ and $BG'$.
The correspondence of the path-components was already known by M. Fuchs \cite{Fuc65} for $n=\infty$.

As an application of this result, we prove that the evaluation fiber sequence extends as
\begin{align*}
\overline{\Map}_0(B_nG,BG)\rightarrow\overline{\Map}(B_nG,BG)\rightarrow BG\rightarrow B\mathcal{W}\mathcal{A}_n(G,G;\mathrm{eq}),
\end{align*}
where the spaces $\overline{\Map}_0(B_nG,BG)$ and $\overline{\Map}(B_nG,BG)$ are unions of appropriate path-components in $\Map_0(B_nG,BG)$ and $\Map(B_nG,BG)$, respectively, $\mathcal{A}_n(G,G;\mathrm{eq})$ denotes the space of self-$A_n$-equivalences on $G$, and the functor $\mathcal{W}$ is a kind of ``cofibrant replacement''.
This is a generalization of the well-known extension of the evaluation fiber sequence
\[
	\overline{\Map}_0(X,X)\rightarrow\overline{\Map}(X,X)\rightarrow X\rightarrow B\mathcal{W}\overline{\Map}_0(X,X)\rightarrow B\mathcal{W}\overline{\Map}(X,X),
\]
where $\overline{\Map}_0(X,X)\subset\Map_0(X,X)$ and $\overline{\Map}(X,X)\subset\Map(X,X)$ are the monoids consisting of homotopy equivalences.
This extension can be found in \cite{Got73,May80}.
Moreover, our result gives the maximum extension because the path-component $\Map(B_nG,BG;\iota_n)$ of the inclusion $\iota_n\colon B_nG\to BG$ cannot be delooped in general.
We will give such an example.

This paper is organized as follows.
In Section \ref{section_mapping_spaces}, we collect elementary facts about mapping spaces.
In Section \ref{section_maps_between_cubes}, some maps between cubes are defined, which will be used to describe the topological category $\mathcal{A}_n$.
In Section \ref{section_category_of_A_n-maps}, we define the topological category $\mathcal{A}_n$ of topological monoids and $A_n$-maps.
In Section \ref{section_bar_construction}, we construct the continuous bar construction functor.
In Section \ref{section_mapping_from_projective}, we investigate the mapping spaces from projective spaces and show the above weak equivalence.
The rest is devoted to applications.
In Section \ref{section_evaluation}, we give an extension of the evaluation fiber sequence as above.
In Section \ref{section_appl_comm}, the relation with our result and various higher homotopy commutativities such as $C_k$-spaces by Williams \cite{Wil69}, $C^k$-spaces by Sugawara \cite{Sug60}, $C_k(n)$-spaces by Hemmi--Kawamoto \cite{HK11} and $C(k,\ell)$-spaces by Kishimoto--Kono \cite{KK10} are studied.
In Section \ref{section_appl_gauge}, equivalent conditions for the adjoint bundle of a principal bundle being trivial are given.
In Section \ref{section_appl_cyclic}, we give some application to $T_k^f$-spaces introduced by Iwase--Mimura--Oda--Yoon \cite{IMOY12}.
In particular, we show that a pointed space is a $T_k^f$-space if and only if it is a $C_k^f$-space and, as an example, study when $B\SU(2)$ is a $T_k^f$-space for a map $f\colon S^4\to B\SU(2)$.
In Section \ref{section_homotopypullback}, we make some remarks on a relation of our result and homotopy pullbacks of $A_n$-maps.

The author would like to express his gratitude to Professors Kouyemon Iriye, Norio Iwase, Daisuke Kishimoto and Nobuyuki Oda for fruitful discussions and encouragement for this work.

\section{Preliminaries on mapping spaces}
\label{section_mapping_spaces}
We collect elementary facts on mapping spaces.
We refer to \cite[Section 2.4, 4.2]{Hov99} about the category of compactly generated spaces.
We will work in the categories $\mathbf{CG}$ of compactly generated spaces and $\mathbf{CG}_*$ of pointed ones.
In these categories, a map $f\colon X\rightarrow Y$ is said to be a \textit{weak equivalence} if $f$ induces isomorphisms on homotopy groups with respect to any basepoints $x_0\in X$ and $f(x_0)\in Y$.
\begin{dfn}
For compactly generated spaces $X$ and $Y$, we denote the \textit{mapping space} between $X$ and $Y$ by $\Map(X,Y)$ which consists of continuous maps from $X$ to $Y$ as a set.
We write the subspace of basepoint preserving maps by $\Map_0(X,Y)$ for pointed spaces $X$ and $Y$.
For a pointed map $\phi\colon X\to Y$, the path-component of $\Map(X,Y)$ containing $\phi$ is denoted by $\Map(X,Y;\phi)$.
Similarly, we denote $\Map_0(X,Y;\phi):=\Map(X,Y;\phi)\cap\Map_0(X,Y)$.
Do not confuse it with the path-component of $\Map_0(X,Y)$ containing $\phi$.
Unless otherwise stated, the basepoints of $\Map_0(X,Y)$ and $\Map(X,Y)$ are the constant map.
\end{dfn}
The functors $\Map$ and $\Map_0$ satisfy the following exponential laws:
\begin{align*}
\Map(X,\Map(Y,Z))&\cong\Map(X\times Y,Z),\\
\Map_0(X,\Map_0(Y,Z))&\cong\Map_0(X\wedge Y,Z),
\end{align*}
where $X\wedge Y:=(X\times Y)/(X\times *\cup *\times Y)$ denotes the smash product of $X$ and $Y$.

From this, the evaluation map
\begin{align*}
\Map(X,Y)\times X\rightarrow Y,\quad (f,x)\mapsto f(x)
\end{align*}
is continuous since it is the adjoint map of the identity map $\Map(X,Y)\rightarrow\Map(X,Y)$.
Moreover, the composition
\begin{align*}
\circ\colon\Map(Y,Z)\times\Map(X,Y)\rightarrow\Map(X,Z),\quad (g,f)\mapsto g\circ f
\end{align*}
is continuous since it is the adjoint of the continuous map
\begin{align*}
\Map(Y,Z)\times\Map(X,Y)\times X\rightarrow Z,\quad (g,f,x)\mapsto g(f(x)).
\end{align*}
Similar properties hold for $\Map_0$.

\begin{prp}
\label{prp_equiv_map_sp}
Let $X$ and $X'$ be pointed CW complexes and $Y$ and $Y'$ pointed spaces.
Then, pointed weak equivalences $f\colon X\rightarrow Y$ and $f'\colon X'\rightarrow Y'$ induce the following weak equivalences:
\begin{align*}
f^{\#}&\colon\Map_0(Y,X')\rightarrow\Map_0(X,X'),\\
f^{\#}&\colon\Map(Y,X')\rightarrow\Map(X,X'),\\
f'_{\#}&\colon\Map_0(X,X')\rightarrow\Map_0(X,Y'),\\
f'_{\#}&\colon\Map(X,X')\rightarrow\Map(X,Y').
\end{align*}
\end{prp}

\begin{rem}
This proposition obviously generalizes to the case when $X$ and $X'$ are only assumed to have the pointed homotopy types of CW complexes.
\end{rem}

\begin{dfn}
A pointed space $X$ is said to be \textit{well-pointed} if the inclusion of the basepoint $*\subset X$ has the homotopy extension property.
\end{dfn}

If $X$ is well-pointed and $Y$ is pointed, then the evaluation $\Map(X,Y)\rightarrow Y$ at the basepoint has the homotopy lifting property and its fiber at the basepoint is $\Map_0(X,Y)$.

\begin{ex}
\begin{enumerate}[(i)]
\item
Every pointed CW complex is well-pointed.
\item
It is well-known that a pointed space $X$ is well-pointed if and only if the pair $(X,*)$ is an NDR pair, that is, there exist a map $u\colon X\rightarrow [0,1]$ and a homotopy $h\colon[0,1]\times X\rightarrow X$ such that $u^{-1}(0)=*$ and the following equalities hold:
\begin{align*}
&H(0,x)=x&\text{ for }x\in X,\\
&H(t,*)=*&\text{ for }0\le t\le 1,\\
&H(1,x)=*&\text{ if }u(x)<1.
\end{align*}
Suppose that there exist a map $u\colon X\rightarrow [0,1]$ and a homotopy $h\colon[0,1]\times X\rightarrow X$ as above.
Let $K$ be a compact pointed space.
Define $u'\colon\Map(K,X)\rightarrow [0,1]$ and $H'\colon[0,1]\times\Map(K,X)\rightarrow\Map(K,X)$ by
\begin{align*}
&u'(f)=\max u(f(K)),\\
&H'(t,f)(k)=H(t,f(k)).
\end{align*}
Then $u'$ and $H'$ satisfy the above properties for the pairs $(\Map(K,X),*)$ and $(\Map_0(K,X),*)$.
Therefore, the mapping spaces $\Map(K,X)$ and $\Map_0(K,X)$ are well-pointed.
\end{enumerate}
\end{ex}

For a pointed space $X$, the space $\Omega X:=\Map_0(S^1,X)$ is called the \textit{based loop space} of $X$, which has the pointed homotopy type of a CW complex if $X$ is pointed homotopy equivalent to a CW complex.
By concatenation of loops, $\Omega X$ becomes a homotopy associative $H$-space with homotopy unit.
To make this operation associative and unital, we use the Moore path technique.
The \textit{Moore based path space} $P^{\mathrm{M}}X$ of $X$ is defined by
\begin{align*}
P^{\mathrm{M}}X:=\{(g,\ell)\in\Map([0,\infty),X)\times[0,\infty)\mid g(t)=*\text{ for }t\ge\ell\}.
\end{align*}
There is the evaluation map
\begin{align*}
e\colon P^{\mathrm{M}}X\rightarrow X,\quad e(g,\ell)=g(0),
\end{align*}
which is a Hurewicz fibration.
The fiber over the basepoint is denoted by $\Omega^{\mathrm{M}}X$ and called the \textit{Moore based loop space}.
There is the associative concatenation operation: for $(g,\ell)\in P^{\mathrm{M}}X,(g',\ell')\in\Omega^{\mathrm{M}}X$, $g+g'\colon[0,\infty)\rightarrow X$ is defined by
\begin{align*}
(g+g')(t)=\left\{
\begin{array}{ll}
g(t) & (t\le\ell) \\
g'(t-\ell) & (t\ge\ell)
\end{array}
\right.
\end{align*}
and $(g,\ell)+(g',\ell'):=(g+g',\ell+\ell')\in P^{\mathrm{M}}X$.
This operation gives continuous maps 
\begin{align*}
&+\colon P^{\mathrm{M}}X\times\Omega^{\mathrm{M}}X\rightarrow P^{\mathrm{M}}X,\\
&+\colon\Omega^{\mathrm{M}}X\times\Omega^{\mathrm{M}}X\rightarrow\Omega^{\mathrm{M}}X
\end{align*}
and makes $\Omega^{\mathrm{M}}X$ a topological monoid and $P^{\mathrm{M}}X\rightarrow X$ a principal fibration.

For a pointed map $f\colon A\rightarrow B$ and a pointed space $X$, the cofiber sequence $A\rightarrow B\rightarrow C_f$ induces the homotopy fiber sequence
\begin{align*}
\Map_0(\Sigma A,X)\rightarrow\Map_0(C_f,X)\rightarrow\Map_0(B,X),
\end{align*}
where $\Sigma A$ is the reduced suspension of $A$ and $C_f$ is the reduced mapping cone of $f$.
The canonical pinch maps
\begin{align*}
C_f\rightarrow C_f\vee\Sigma A,\quad\Sigma A\rightarrow\Sigma A\vee\Sigma A
\end{align*}
give a homotopy associative action of $\Map_0(\Sigma A,X)$ on $\Map_0(C_f,X)$.
This also can be replaced by an associative one as follows.
Let us consider the pullback
\begin{align*}
\xymatrix{
\Map_0^{\mathrm{M}}(C_f,X) \ar[r] \ar[d] & P^{\mathrm{M}}\Map_0(A,X) \ar[d]^-{e} \\
\Map_0(B,X) \ar[r]_-{f^{\#}} & \Map_0(A,X).
}
\end{align*}
Then we have a principal fibration
\begin{align*}
\Omega^{\mathrm{M}}\Map_0(A,X)\rightarrow\Map_0^{\mathrm{M}}(C_f,X)\rightarrow\Map_0(B,X),
\end{align*}
which is naturally equivalent to the above homotopy fiber sequence.

\section{Certain maps between cubes}
\label{section_maps_between_cubes}
Consider the closed interval $[0,\infty]=[0,\infty)\cup\{\infty\}$ homeomorphic to the unit interval.
Define the following maps:
\begin{align*}
&\delta_k^t\colon[0,\infty]^{\times(i-1)}\rightarrow [0,\infty]^{\times i},\quad \delta_k^t(t_1,\ldots,t_{i-1})=(t_1,\ldots,t_{k-1},t,t_k,\ldots,t_{i-1}),\\
&\sigma_k:[0,\infty]^{\times(i-1)}\rightarrow[0,\infty]^{\times(i-2)},\quad \sigma_k(t_1,\ldots,t_{i-1})=\left\{
\begin{array}{ll}
(t_2,\ldots,t_{i-1}) & (k=1) \\
(t_1,\ldots,t_{k-2},\max\{t_{k-1},t_k\},t_{k+1},\ldots,t_{i-1}) & (1<k<i) \\
(t_1,\ldots,t_{i-2}) & (k=i)
\end{array}
\right.
\end{align*}
for $1\le k\le i$.

We also give a cubical partition of the cube $[0,\infty]^{\times(i-1)}$.
For a multi-index $\mathbf{i}=(i_1,\ldots,i_r)$ consisting of positive integers with $i_1+\cdots+i_r=i$ and $\ell\in[0,\infty)$, define the map
\begin{align*}
&\gamma_{\mathbf{i}}^\ell:[0,\infty]^{\times(r-1)}\times[0,\ell]^{\times(i_1-1)}\times\cdots\times[0,\ell]^{\times(i_r-1)}\rightarrow[0,\infty]^{\times(i-1)},\\
&\gamma_{\mathbf{i}}^\ell(\mathbf{t};\mathbf{s}_1,\ldots,\mathbf{s}_r)=(\mathbf{s}_1,t_1+\ell,\mathbf{s}_2,t_2+\ell,\ldots,t_{r-1}+\ell,\mathbf{s}_r)
\end{align*}
for $\mathbf{t}=(t_1,\ldots,t_{r-1})\in[0,\infty]^{\times(r-1)}$ and $\mathbf{s}_k\in[0,\ell]^{\times(i_k-1)}$.
For example, the images on $[0,\infty]^{\times 2}$ is depicted in Figure \ref{fig_gamma1-1}.
\begin{figure}
\centering
\includegraphics[width=5cm,clip]{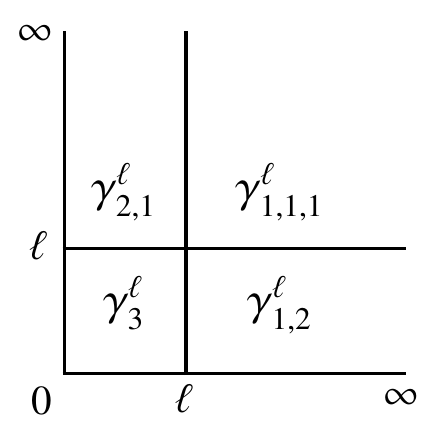}
\caption{}
\label{fig_gamma1-1}
\end{figure}

\begin{lem}
\label{lem_identities_gamma}
The following identities hold:
\begin{align*}
\begin{array}{cll}
\text{(i)}
	&\delta_{i_1+\cdots+i_{k-1}+m}^t(\gamma_{(i_1,\ldots,i_r)}^\ell(\mathbf{t};\mathbf{s}_1,\ldots,\mathbf{s}_r))
	&
	\\
&=\gamma_{(i_1,\ldots,i_{k-1},i_k+1,i_{k+1},\ldots,i_r)}^\ell(\mathbf{t};\mathbf{s}_1,\ldots,\mathbf{s}_{k-1},\delta_m^t(\mathbf{s}_k),\mathbf{s}_{k+1},\ldots,\mathbf{s}_r)
	&\text{for }1\le m\le i_k\text{ and }t\le\ell,
	\\
\text{(ii)}
	&\delta_{i_1+\cdots+i_{k-1}+m}^{\ell+t}(\gamma_{(i_1,\ldots,i_r)}^\ell(\mathbf{t};\mathbf{s}_1,\ldots,\mathbf{s}_r))
	&
	\\
&=\gamma_{(i_1,\ldots,i_{k-1},m-1,i_k-m,i_{k+1},\ldots,i_r)}^\ell(\delta_k^t(\mathbf{t});\mathbf{s}_1,\ldots,\mathbf{s}_{k-1},\sigma_m^{i_k-m}(\mathbf{s}_k),\sigma_1^{m-1}(\mathbf{s}_k),\mathbf{s}_{k+1},\ldots,\mathbf{s}_r)
	&\text{for }1\le m\le i_k\text{ and }t\ge 0,
	\\
\text{(iii)}
	&\sigma_{i_1+\cdots+i_{k-1}+m}(\gamma_{(i_1,\ldots,i_r)}^\ell(\mathbf{t};\mathbf{s}_1,\ldots,\mathbf{s}_r))
	&
	\\
&=\gamma_{(i_1,\ldots,i_{k-1},i_k-1,i_{k+1},\ldots,i_r)}^\ell(\mathbf{t};\mathbf{s}_1,\ldots,\mathbf{s}_{k-1},\sigma_m(\mathbf{s}_k),\mathbf{s}_{k+1},\ldots,\mathbf{s}_r)
	&\text{for }1\le m\le i_k,
	\\
\text{(iv)}
	&\gamma_{(i_1,\ldots,i_r)}^{\ell+\ell'}(
		\mathbf{u};
		\gamma_{(i_{1,1},\ldots,i_{1,q_1})}^{\ell}(
		\mathbf{t}_1;\mathbf{s}_{1,1},\ldots,\mathbf{s}_{1,q_1}
		),
		\ldots,
		\gamma_{(i_{r,1},\ldots,i_{r,q_r})}^{\ell}(
			\mathbf{t}_r;\mathbf{s}_{r,1},\ldots,\mathbf{s}_{r,q_r}
		)
	)
	&
	\\
&=\gamma_{(i_{1,1},\ldots,i_{1,q_1},\ldots,i_{r,1},\ldots,i_{r,q_r})}^{\ell}(
		\gamma_{(i_1,\ldots,i_r)}^{\ell'}(
			\mathbf{u};\mathbf{t}_1,\ldots,\mathbf{t}_r
		);
		\mathbf{s}_{1,1},\ldots,\mathbf{s}_{1,q_1},\ldots,\mathbf{s}_{r,1},\ldots,\mathbf{s}_{r,q_r}
		).
	&
\end{array}
\end{align*}
\end{lem}

\section{The topological category of $A_n$-maps between topological monoids}
\label{section_category_of_A_n-maps}

In this section, we formulate the topological category of topological monoids and $A_n$-maps between them.
It is known that there exists a quasicategory of $A_n$-spaces and $A_n$-maps between them by \cite{Tsu15}, which is justified by the results of Boardman--Vogt in \cite{BV73}.
Hence, using the technique of quasicategories, one can obtain the topological category of $A_n$-spaces from this quasicategory.
But our approach is different from this and rather elementary.

As we have seen in Section \ref{section_mapping_spaces}, the source space of a mapping space should be a CW complex if one wants to consider a well-behaved mapping space.
Similarly, considering maps between topological monoids, we should consider \textit{grouplike} ones as target spaces.

\begin{dfn}
A topological monoid $G$ is said to be \textit{grouplike} if the monoid $\pi_0(G)$ is a group with respect to the multiplication induced from the monoid structure of $G$.
\end{dfn}

For example, any Moore based loop space $\Omega^{\mathrm{M}}X$ is grouplike.

The following lemma is proved by induction on the dimension of the skeletons of $A$.

\begin{lem}
\label{lem_grouplike_inversion}
Let $G$ be a grouplike topological monoid and $A$ a pointed space of the pointed homotopy type of a CW complex.
Then, for any pointed map $\alpha\colon A\rightarrow G$, there exists a map $\alpha'\colon A\rightarrow G$ such that the maps
\begin{align*}
A\rightarrow G,\quad a\mapsto\alpha(a)\alpha'(a)\quad\text{and}\quad a\mapsto\alpha'(a)\alpha(a)
\end{align*}
are pointed homotopic to the constant map.
\end{lem}

This lemma leads the following proposition.

\begin{prp}
\label{prp_grouplike_hofib_map}
Let $i\colon A\to B$ be an inclusion having the homotopy extension property between spaces of the homotopy types of CW complexes and $\alpha\colon B\to G$ a map to a grouplike topological monoid.
Suppose that $\Map(A,G)$ and $\Map(B,G)$ are pointed at the maps $\alpha|_A$ and $\alpha$, respectively.
Then the inclusion of the homotopy fiber of $i^{\#}\colon\Map(B,G)\to\Map(A,G)$ is equivalent to the map $\Map_0(B/A,G)\to\Map(B,G)$ given by the multiplication $\beta\mapsto\alpha\beta$.
More precisely, the map $i^{\#}$ is a principal homotopy fibration with fiber $\Map_0(B/A,G)$.
\end{prp}

\begin{proof}
If $\alpha$ is the constant map $0$, then the proposition is obvious.
By Lemma \ref{lem_grouplike_inversion}, there is a map $\alpha'\colon B\to G$ such that $\alpha\alpha'$ and $\alpha'\alpha$ is homotopic to the constant map.
Then there is a commutative diagram
\[
\xymatrix{
	\Map(B,G) \ar[r]^-{i^{\#}} \ar[d]_-{\alpha\cdot} &
		\Map(A,G) \ar[d]^-{(\alpha|_A)\cdot} \\
	\Map(B,G) \ar[r]_-{i^\#} &
		\Map(A,G)
}
\]
such that the vertical maps are given by multiplying $\alpha$ from the left and are homotopy equivalences.
By these equivalences, the proposition follows from the case when $\alpha=0$.
\end{proof}

Now, we define $A_n$-maps.
Our definition is slightly different from Stasheff's original one in \cite{Sta63b}.
Since we want to make the composition of $A_n$-maps associative and unital, the Moore type definition is adopted.
 
\begin{dfn}
Let $G$ and $G'$ be topological monoids and $f\colon G\to G'$ a pointed map.
A family $\{f_i\colon [0,\infty]^{\times(i-1)}\times G^{\times i}\to G'\}_{i=1}^n$ of maps is said to be an \textit{$A_n$-form} of size $\ell\in[0,\infty)$ if the following conditions hold:
\begin{align*}
\begin{array}{cll}
\text{(i)} &
f_1=f, &
\\
\text{(ii)} &
f_i(\delta_k^0(\mathbf{t});g_1,\ldots,g_i)=f_{i-1}(\mathbf{t};g_1,\ldots,g_kg_{k+1},\ldots,g_i),
\\
\text{(iii)} &
f_i(\delta_k^t(\mathbf{t});g_1,\ldots,g_i)=f_{k-1}(\sigma_k^{i-k}(\mathbf{t});g_1,\ldots,g_k)f_{i-k}(\sigma_1^{k-1}(\mathbf{t});g_{k+1},\ldots,g_i) & \text{for }t\ge\ell,
\\
\text{(iv)} &
f_i(\mathbf{t};g_1,\ldots,g_{k-1},*,g_{k+1},\ldots,g_i)=f_{i-1}(\sigma_k(\mathbf{t});g_1,\ldots,g_{k-1},g_{k+1},\ldots,g_i). &
\end{array}
\end{align*}
The triple $f=(f,\{f_i\}_{i=1}^n,\ell)$ is called an \textit{$A_n$-map}.
In particular, if the size $\ell$ of the $A_n$-form of $f$ is 0, $f$ is a homomorphism.
We denote the space of $A_n$-maps between topological monoids $G$ and $G'$ by $\mathcal{A}_n(G,G')$.
\end{dfn}

We consider that an $A_1$-map is a pair $(f,\ell)$ consisting of a pointed map $f$ and a meaningless number $\ell$.
Unless otherwise stated, we assign a homomorphism the $A_n$-form of size $0$.

\begin{dfn}\label{dfn_composition}
Let $G$, $G'$ and $G''$ be topological monoids.
For $A_n$-maps $f=(f,\{f_i\}_{i=1}^n,\ell):G\to G'$ and $f'=(f',\{f'_i\}_{i=1}^n,\ell'):G'\to G''$, the composition $f'\circ f=(f'\circ f,\{F_i\}_{i=1}^n,\ell+\ell')$ is defined as follows.
For a multi-index $\mathbf{i}=(i_1,\ldots,i_r)$ with $i_1+\cdots+i_r=i$, $\mathbf{t}\in[0,\infty]^{\times(r-1)}$, $\mathbf{s}_k\in[0,\ell]^{\times(i_k-1)}$ and $\mathbf{g}_k\in G^{\times i_k}$, define
\begin{align*}
F_i(\gamma_{\mathbf{i}}^{\ell}(\mathbf{t};\mathbf{s}_1,\ldots,\mathbf{s}_r);\mathbf{g}_1,\ldots,\mathbf{g}_r)=f'_r(\mathbf{t};f_{i_1}(\mathbf{s}_1;\mathbf{g}_1),\ldots,f_{i_r}(\mathbf{s}_r;\mathbf{g}_r)).
	\end{align*}
\end{dfn}

Using the identities (i), (ii) and (iii) in Lemma \ref{lem_identities_gamma}, it is verified that $\{F_i\}$ is an $A_n$-form on $f'\circ f$ of size $\ell+\ell'$.

This composition defines a continuous map
\begin{align*}
	\circ:\mathcal{A}_n(G',G'')\times\mathcal{A}_n(G,G')\to\mathcal{A}_n(G,G'').
\end{align*}
By the identity (iv) in Lemma \ref{lem_identities_gamma}, it is an associative and unital operation, where the identity in $\mathcal{A}_n(G,G)$ is given by the identity map $\mathrm{id}_G$.
Thus we obtain a topological category as follows.

\begin{dfn}
\label{dfn_category_of_A_n}
The topological category $\mathcal{A}_n$ consists of topological monoids as objects and $\mathcal{A}_n(G,G')$ as a morphism space between each pair $G$ and $G'$, where the composition is given as Definition \ref{dfn_composition}.
\end{dfn}

If the underlying map $f_1$ of $f\in\mathcal{A}_n(G,G')$ is a weak equivalence, $f$ is said to be a \textit{weak $A_n$-equivalence}.
The \textit{homotopy category} $\pi_0\mathcal{A}_n$ of $\mathcal{A}_n$ is the category whose objects are the same as $\mathcal{A}_n$ and the morphism set between $G$ and $G'$ is defined by $(\pi_0\mathcal{A}_n)(G,G')=\pi_0(\mathcal{A}_n(G,G'))$. 

\begin{dfn}
Let $G$ and $G'$ be topological monoids and $f=(f,\{f_i\}_{i=1}^n,\ell):G\to G'$ an $A_n$-map.
For a left $G$-space $X$, a left $G'$-space $X'$ and a map $\phi:X\to X'$, a family $\{\phi_i:[0,\infty]^i\times G^i\to G'\}_{i=0}^n$ of maps is said to be an \textit{$A_n$-form} if the following conditions hold:
\begin{align*}
\begin{array}{cll}
\text{(i)} &
\phi_0=\phi, &
\\
\text{(ii)} &
\phi_i(\delta_k^0(\mathbf{t});g_1,\ldots,g_i,x)=\left\{
	\begin{array}{ll}
		\phi_{i-1}(\mathbf{t};g_1,\ldots,g_kg_{k+1},\ldots,g_i,x) & (k<i) \\
		\phi_{i-1}(\mathbf{t};g_1,\ldots,g_{i-1},g_ix) & (k=i),
	\end{array}
\right.
\\
\text{(iii)} &
\phi_i(\delta_k^t(\mathbf{t});g_1,\ldots,g_i,x)=f_k(\sigma_k^{i-k+1}(\mathbf{t});g_1,\ldots,g_k)\phi_{i-k}(\sigma_1^{k-1}(\mathbf{t});g_{k+1},\ldots,g_i,x)
& \text{for }t\ge\ell,
\\
\text{(iv)} &
f_i(\mathbf{t};g_1,\ldots,g_{k-1},*,g_{k+1},\ldots,g_i,x)=f_{i-1}(\sigma_k(\mathbf{t});g_1,\ldots,g_{k-1},g_{k+1},\ldots,g_i,x). &
\end{array}
\end{align*}
The quintuple $\phi=(f,\{f_i\}_{i=1}^n,\ell,\phi,\{\phi_i\}_{i=1}^n)$ is called an \textit{$A_n$-equivariant map} through the $A_n$-map $f=(f,\{f_i\}_{i=1}^n,\ell)$.
In particular, if the size $\ell$ of the $A_n$-form of $f$ is 0, $\phi$ is an ordinary equivariant map.
We denote the space of $A_n$-equivariant maps between a left $G$-space $X$ and a left $G'$-space $X'$ by $\mathcal{A}_n^{\mathrm{L}}((G,X),(G',X'))$.
Similarly, we define $A_n$-equivariant maps between a right $G$-space $X$ and a right $G'$-space $X'$, and the corresponding mapping space is denoted by $\mathcal{A}_n^{\mathrm{R}}((X,G),(X',G'))$.
\end{dfn}
\begin{dfn}
The topological categories $\mathcal{A}_n^{\mathrm{L}}$ and $\mathcal{A}_n^{\mathrm{R}}$ are defined as well as $\mathcal{A}_n$ in Definition \ref{dfn_category_of_A_n}.
We define the projection functors $\mathcal{A}_n^{\mathrm{L}}\to\mathcal{A}_n$ and $\mathcal{A}_n^{\mathrm{R}}\to\mathcal{A}_n$ by $(G,X)\mapsto G$ and $(X,G)\mapsto G$, respectively.
\end{dfn}

In the following, we will prove various properties about the mapping spaces of $\mathcal{A}_n$.
It will be convenient to consider the deformation retract
\begin{align*}
\mathcal{A}_n^1(G,G')\subset\mathcal{A}_n(G,G')
\end{align*}
consisting of $A_n$-forms of size $\ge 1$.
Assume that $G$ is well-pointed and of the homotopy type of a CW complex, and that $G'$ is grouplike.
Note that there is a homotopy pullback diagram
\[
\xymatrix{
	\mathcal{A}_n^1(G,G') \ar[r] \ar[d] &
		\Map([0,1]^{\times(n-1)}\times G^{\times n},G') \ar[d] \\
	\mathcal{A}_{n-1}^1(G,G') \ar[r] &
		\Map(\partial[0,1]^{\times(n-1)}\times G^{\times n}\cup[0,1]^{\times(n-1)}\times T_nG,G'),
}
\]
where the left vertical map is the forgetful map and $T_nG\subset G^{\times n}$ denotes the fat wedge of $G$:
\begin{align*}
T_nG:=\{(g_1,\ldots,g_n)\in G^{\times n}\mid g_k=*\text{ for some }k\}.
\end{align*}
Here, the homotopy fiber is determined as $\Map_0(\Sigma^{n-1}G^{\wedge n},G')$ by using Proposition \ref{prp_grouplike_hofib_map} and does not depend on the choice of the basepoint of $\mathcal{A}_n^1(G,G')$.
Moreover, the homotopy fiber has an action
\begin{align*}
\mathcal{A}_n^1(G,G')\times\Map_0(\Sigma^{n-1}G^{\wedge n},G')\rightarrow\mathcal{A}_n^1(G,G')
\end{align*}
defined as follows, which gives the structure of a principal fibration.
For $\alpha\in\Map_0(\Sigma^{n-1}G^{\wedge n},G')$ and $f=(f,\{f_i\}_i,\ell)\in\mathcal{A}_n^1(G,G')$, the action is defined by $f\cdot\alpha:=(f,\{f_i\cdot\alpha\}_i,\ell)$ such that
\begin{align*}
(f_i\cdot\alpha)(t_1,\ldots,t_{i-1};\mathbf{g})=\left\{
\begin{array}{ll}
f_i(t_1,\ldots,t_{i-1};\mathbf{g}) & (i<n) \\
f_n(t_1,\ldots,t_{n-1};\mathbf{g})\alpha(\min\{t_1,1\},\ldots,\min\{t_{n-1},1\};\mathbf{g}) & (i=n)
\end{array}
\right. ,
\end{align*}
where $\alpha$ is considered as the map $[0,1]^{\times(n-1)}\times G^{\times n}\rightarrow G$.

\begin{prp}
\label{prp_equiv_mappingsp_A_n}
Let $G,G',H,H'$ be topological monoids.
Suppose that $G$ and $H$ are well-pointed and homotopy equivalent to CW complexes.
Then the compositions of weak $A_n$-equivalences $f=(f,\{f_i\},\ell):G\rightarrow H$ and $f'=(f',\{f_i'\},\ell'):G'\rightarrow H'$ induce the following weak equivalences:
\begin{align*}
f^{\#}&:\mathcal{A}_n(H,G')\xrightarrow{\simeq}\mathcal{A}_n(G,G'),\\
f'_{\#}&:\mathcal{A}_n(H,G')\xrightarrow{\simeq}\mathcal{A}_n(H,H').
\end{align*}
\end{prp}

\begin{proof}
It is sufficient to prove that the maps
\begin{align*}
f^{\#}&:\mathcal{A}_n^1(H,G')\xrightarrow{\simeq}\mathcal{A}_n^1(G,G'),\\
f'_{\#}&:\mathcal{A}_n^1(H,G')\xrightarrow{\simeq}\mathcal{A}_n^1(H,H')
\end{align*}
are weak equivalences.
For $n=1$, it follows from Proposition \ref{prp_equiv_map_sp}.
Suppose that the claim is true for $n-1$.
The map $f^{\#}\colon\mathcal{A}_n^1(H,G')\to\mathcal{A}_n^1(G,G')$ is recognized as the map obtained by taking the homotopy pullback along the horizontal direction of the homotopy commutative diagram
\[
\xymatrix{
	\mathcal{A}_{n-1}^1(H,G') \ar[r] \ar[d]_-{f^{\#}} &
		\Map(\partial[0,1]^{\times(n-1)}\times H^{\times n}\cup[0,1]^{\times(n-1)}\times T_nH,G') \ar[d]^-{f_1^{\#}} &
		\Map([0,1]^{\times(n-1)}\times H^{\times n},G') \ar[l] \ar[d]^-{f_1^{\#}} \\
	\mathcal{A}_{n-1}^1(G,G') \ar[r] &
		\Map(\partial[0,1]^{\times(n-1)}\times G^{\times n}\cup[0,1]^{\times(n-1)}\times T_nG,G') &
		\Map([0,1]^{\times(n-1)}\times G^{\times n},G') \ar[l]
}
\]
such that the left square commutes only up to the homotopy defined by $\{f_i\}$.
Since the vertical maps are weak equivalences, the resulting map $f^{\#}\colon\mathcal{A}_n^1(H,G')\to\mathcal{A}_n^1(G,G')$ is a weak equivalence as well.
 
For $f'_{\#}$, the claim similarly follows from the diagram
\[
\xymatrix{
	\mathcal{A}_{n-1}^1(H,G') \ar[r] \ar[d]_-{f'_{\#}} &
		\Map(\partial[0,1]^{\times(n-1)}\times H^{\times n}\cup[0,1]^{\times(n-1)}\times T_nH,G') \ar[d]^-{f_1'{}_{\#}} &
		\Map([0,1]^{\times(n-1)}\times H^{\times n},G') \ar[l] \ar[d]^-{f_1'{}_{\#}} \\
	\mathcal{A}_{n-1}^1(H,H') \ar[r] &
		\Map(\partial[0,1]^{\times(n-1)}\times H^{\times n}\cup[0,1]^{\times(n-1)}\times T_nH,H') &
		\Map([0,1]^{\times(n-1)}\times H^{\times n},H'). \ar[l]
}
\]
\end{proof}

\begin{cor}
\label{cor_inverse_of_A_n-equivalence}
If $G$ and $G'$ are well-pointed topological monoids of the homotopy types of CW complexes, then every weak $A_n$-equivalence $G\rightarrow G'$ has an inverse in the homotopy category $\pi_0\mathcal{A}_n$.
\end{cor}

\section{Subdivided bar construction functor}
\label{section_bar_construction}
Denote the $i$-dimensional simplex by $\Delta^i$.
Let $X$ be a right $G$-space and $Y$ be a left $G$-space for a topological monoid $G$.
Then, the \textit{$n$-th bar construction} $B_n(X,G,Y)$ is defined as
\[
	B_n(X,G,Y)=\left.\left(\coprod_{0\le i\le n}\Delta^i\times X\times G^{\times i}\times Y\right)\right/\sim
\]
for an appropriate simplicial relation $\sim$.

For our use, it is convenient to replace $\Delta^i$ by the cubical subdivision $\mathcal{Q}_i$ such that
\[
	\mathcal{Q}_i=\{ (t_0,\ldots,t_i)\in [0,\infty]^{\times(i+1)}\mid t_k=\infty\text{ for some }k \}.
\]
The maps $\delta_k^0$, $\sigma_k$ and $\gamma_{\mathbf{i}}^\ell$ induce the following maps:
\begin{align*}
	&\delta_k^0\colon\mathcal{Q}_i\to\mathcal{Q}_{i-1}&&	\qquad\text{for }k=0,\ldots,i,\\
	&\sigma_k\colon\mathcal{Q}_{i-1}\to\mathcal{Q}_i&&		\qquad\text{for }k=0,\ldots,i,\\
	&\gamma_{\mathbf{i}}^\ell\colon\mathcal{Q}_{r-1}\times[0,\ell]^{\times (i_0-1)}\times\cdots\times[0,\ell]^{\times (i_r-1)}\to\mathcal{Q}_{i_0+\cdots+i_r-1}&&		\qquad\text{for }\mathbf{i}=(i_0,\ldots,i_r).
\end{align*}
Here, we shift the parametrization to $(t_0,\ldots,t_i)$.
For example, $\delta_k^0(t_0,\ldots,t_i)=(t_0,\ldots,t_{k-1},0,t_k,\ldots,t_i)$.
Through a homeomorphism $[0,\infty]\cong[0,1]$, a natural homeomorphism $\mathcal{Q}_i\to\Delta^i$ for each $i$ is given by
\[
	(t_0,\ldots,t_i)(\in[0,1]^{\times(i+1)})\mapsto(t_0/T,\ldots,t_i/T)
\]
such that $T=t_0+\cdots+t_i$, which commutes with the boundary and degeneracy operators.
For $i=2$, see Figure \ref{fig_Q2}.

\begin{figure}
\centering
\includegraphics[width=8cm,clip]{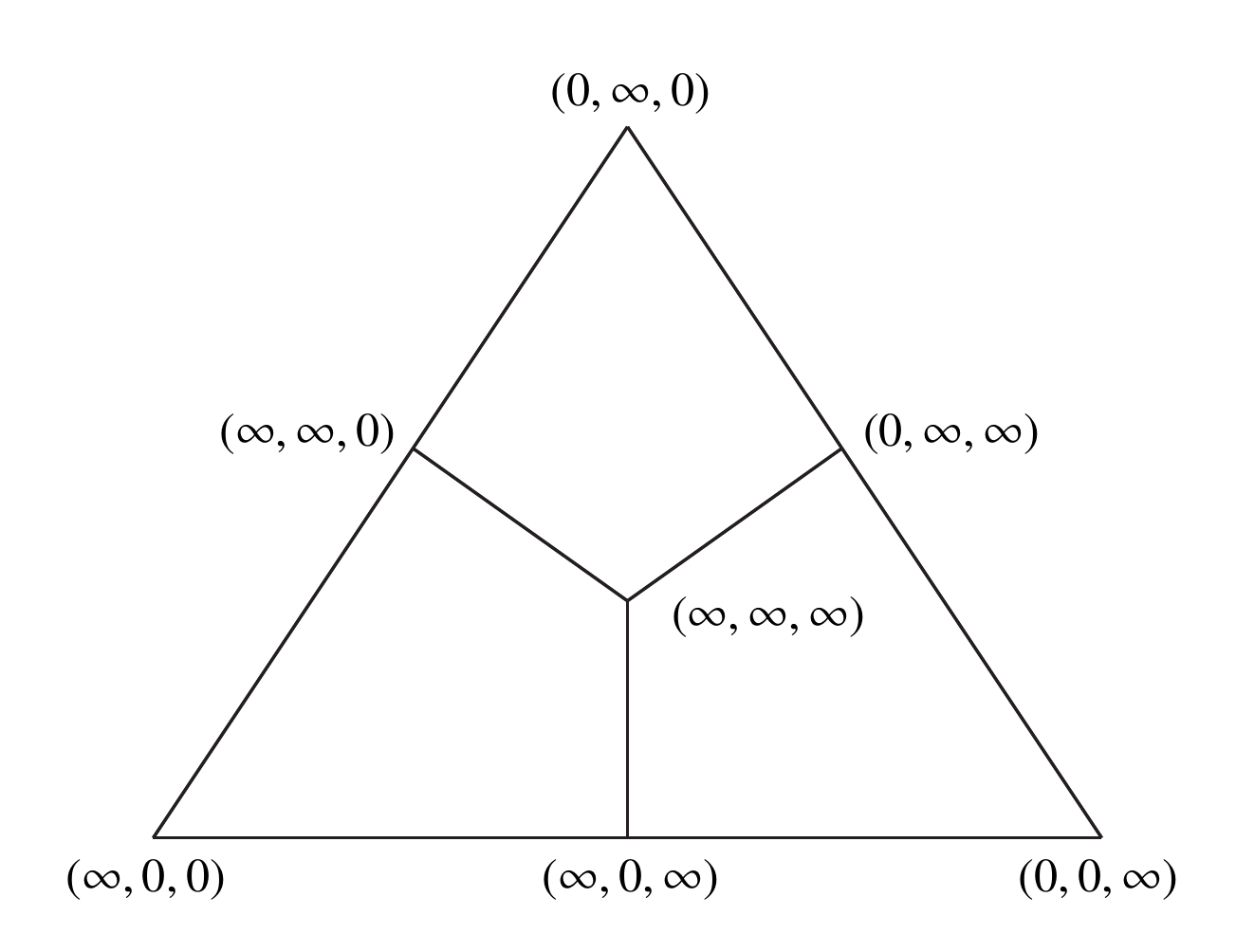}
\caption{}
\label{fig_Q2}
\end{figure}

Consider the topological category $\mathcal{A}_n^{\mathrm{R}}\times_{\mathcal{A}_n}\mathcal{A}_n^{\mathrm{L}}$ defined by the fiber product.
\begin{dfn}
Let $(X,G,Y)\in\mathcal{A}_n^{\mathrm{R}}\times_{\mathcal{A}_n}\mathcal{A}_n^{\mathrm{L}}$.
Then the \textit{$n$-th bar construction} $B_n(X,G,Y)$ is defined by
\begin{align*}
	B_n(X,G,Y)=\left.\left(\coprod_{0\le i\le n}\mathcal{Q}_i\times X\times G^{\times i}\times Y\right)\right/\sim
\end{align*}
with the following identifications:
\begin{align*}
\begin{array}{cl}
\text{(i)} &
(\delta_k^0(\mathbf{t});x,g_1,\ldots,g_i,y)\sim\left\{
	\begin{array}{ll}
		(\mathbf{t};xg_1,g_2,\ldots,g_i,y)					& (k=0) \\
		(\mathbf{t};x,g_1,\ldots,g_{k-1}g_k,\ldots,g_i,y)	& (0<k<i) \\
		(\mathbf{t};x,g_1,\ldots,g_{i-1},g_iy)				& (k=i),
	\end{array}
	\right.
\\
\text{(ii)} &
(\mathbf{t};x,g_1,\ldots,g_{k-1},*,g_{k+1},\ldots,g_i,y)\sim(\sigma_k(\mathbf{t});x,g_1,\ldots,g_{k-1},g_{k+1},\ldots,g_i,y).
\end{array}
\end{align*}
The maps $\iota_{n_1}^{n_2}:B_{n_1}(X,G,Y)\rightarrow B_{n_2}(X,G,Y)$ for $n_1\le n_2$ and $\iota_n:B_n(X,G,Y)\to B(X,G,Y):=B_\infty(X,G,Y)$ denote the inclusions.
\end{dfn}

In fact, this construction induces a continuous functor
\begin{align*}
	B_n:\mathcal{A}_n^{\mathrm{R}}\times_{\mathcal{A}_n}\mathcal{A}_n^{\mathrm{L}}\to\mathbf{CG}
\end{align*}
as follows.
Let $(\phi,f,\psi):(X,G,Y)\to(X',G',Y')$ be a map of size $\ell\in[0,\infty)$ in $\mathcal{A}_n^{\mathrm{R}}\times_{\mathcal{A}_n}\mathcal{A}_n^{\mathrm{L}}$.
Define a map $B_n(\phi,f,\psi):B_n(X,G,Y)\rightarrow B_n(X',G',Y')$ by
\begin{align*}
&B_n(\phi,f,\psi)(\gamma_{\mathbf{i}}^{\ell}(\mathbf{t};\mathbf{s}_0,\ldots,\mathbf{s}_r);x,\mathbf{g}_0,\ldots,\mathbf{g}_r,y)\\
&=[\mathbf{t};\psi_{i_0}(\mathbf{s}_0;x,\mathbf{g}_0),f_{i_1}(\mathbf{s}_1;\mathbf{g}_1),\ldots,f_{i_{r-1}}(\mathbf{s}_{r-1};\mathbf{g}_{r-1}),\psi_{i_r}(\mathbf{s}_r;\mathbf{g}_r,y)]
\end{align*}
for a multi-index $\mathbf{i}=(i_0+1,i_1,\ldots,i_{r-1},i_r+1)$ with $i_0+\cdots+i_r=i$, $\mathbf{t}\in\mathcal{Q}_{r-1}$, $\mathbf{s}_0\in[0,\ell]^{\times i_0}$, $\mathbf{s}_r\in[0,\ell]^{\times i_r}$, $\mathbf{s}_k\in[0,\ell]^{\times(i_k-1)}$ for $0<k<r$, $x\in X$, $\mathbf{g}_k\in G^{\times i_k}$ and $y\in Y$.
Comparing with the definition of compositions of maps in $\mathcal{A}_n^{\mathrm{R}}\times_{\mathcal{A}_n}\mathcal{A}_n^{\mathrm{L}}$, it is straightforward to see that this construction gives a continuous functor $B_n:\mathcal{A}_n^{\mathrm{R}}\times_{\mathcal{A}_n}\mathcal{A}_n^{\mathrm{L}}\to\mathbf{CG}$.

\begin{dfn}
For a topological monoid $G$, the spaces $B_nG:=B_n(*,G,*)$ and $BG:=B_\infty G=B(*,G,*)$ are called the \textit{$n$-th projective space} and the \textit{classifying space} of $G$, respectively.
We denote $E_nG:=B_n(\ast,G,G)$ and $EG:=B(\ast,G,G)$ and often consider the canonical projection $E_nG\to B_nG$.
\end{dfn}

\begin{rem}
Projective spaces and classifying spaces are always path-connected.
\end{rem}

Now we collect several technical lemmas.
Though most of them are well-known, we give a proof using the cubical subdivision $\mathcal{Q}_i$ for consistency.

\begin{lem}
\label{lem_B(X,G,G)}
Let $G$ be a topological monoid, $X$ a right $G$-space, and $Y$ a left $G$-space.
Then the following maps are deformation retractions of the canonical inclusions $X\subset B(X,G,G)$ and $Y\subset B(G,G,Y)$, respectively:
\begin{align*}
& B(X,G,G)\rightarrow X,
	&& [\mathbf{t};x,g_1,\ldots,g_i]\mapsto xg_1\cdots g_i,\\
& B(G,G,Y)\rightarrow Y,
	&& [\mathbf{t};g_1,\ldots,g_i,y]\mapsto g_1\cdots g_iy.
\end{align*}
\end{lem}

\begin{proof}
Define a homotopy $\kappa_i:\mathcal{Q}_1\times\mathcal{Q}_i\rightarrow\mathcal{Q}_{i+1}$ of $\mathcal{Q}_{i+1}$ for each $i$ by
\[
	\kappa_i(s_0,s_1;t_0,\ldots,t_i)=\left(\frac{s_0t_0}{1+s_0+t_0},\ldots,\frac{s_0t_i}{1+s_0+t_i},s_1\right),
\]
which satisfies the following conditions:
\begin{align*}
&\kappa_i(\delta_1^0;\mathbf{t})=\delta_{i+1}^0(\mathbf{t}),\\
&\kappa_i(\delta_0^0;\mathbf{t})=(0,\ldots,0,\infty),\\
&\kappa_i(\mathbf{s};\delta_k^0(\mathbf{t}))=\delta_k^0(\kappa_{i-1}(\mathbf{s};\mathbf{t}))&&\text{ for }0\le k\le i,\\
&\kappa_i(\mathbf{s};\sigma_k(\mathbf{t}))=\sigma_k(\kappa_{i+1}(\mathbf{s};\mathbf{t}))&&\text{ for }0\le k\le i.
\end{align*}
Note that, for any $[\mathbf{t};x,\mathbf{g}]\in B(X,G,G)$, we have an equality
\begin{align*}
[\mathbf{t};x,\mathbf{g}]=[\delta_i^0(\mathbf{t});x,\mathbf{g},*].
\end{align*}
Then the homotopy $H\colon\mathcal{Q}_1\times B(X,G,Y)\to B(X,G,Y)$ defined by
\[
	H(\mathbf{s},[\mathbf{t};x,\mathbf{g}])=[\kappa_i(\mathbf{s},\mathbf{t});x,\mathbf{g},\ast]
\]
is a deformation of $B(X,G,Y)$ to $X$.
One can prove similarly for $Y\subset B(G,G,Y)$.
\end{proof}

\begin{lem}
\label{lem_homotopy_invariance_of_B}
Let $G$ and $G'$ be well-pointed topological monoids, $X$ and $X'$ right $G$-spaces, and $Y$ and $Y'$ left $G$-spaces.
Then the following hold.
\begin{enumerate}[(i)]
\item
The inclusion $B_{n-1}(X,G,Y)\subset B_n(X,G,Y)$ has the homotopy extension property.
\item
A map $(\phi,f,\psi):(X,G,Y)\rightarrow(X',G',Y')$ in $\mathcal{A}_n^{\mathrm{R}}\times_{\mathcal{A}_n}\mathcal{A}_n^{\mathrm{L}}$ induces a weak equivalence $B_n(X,G,Y)\rightarrow B_n(X',G',Y')$ if the underlying maps of $\phi$, $f$ and $\psi$ are weak equivalences.
\end{enumerate}
\end{lem}

\begin{proof}
There is a pushout diagram
\begin{align*}
\xymatrix{
	\partial\mathcal{Q}_n\times X\times G^{\times n}\times Y\cup\mathcal{Q}_n\times X\times T_nG\times Y \ar[r] \ar[d]
		& B_{n-1}(X,G,Y) \ar[d]^{\iota_{n-1}^n} \\
	\mathcal{Q}_n\times X\times G^{\times n}\times Y \ar[r]
		& B_n(X,G,Y).
}
\end{align*}
Then the left vertical arrow has the homotopy extension property since so do the inclusions $\partial\mathcal{Q}_n\subset\mathcal{Q}_n$ and $*\subset G$.
This implies the assertion (i).
For the assertion (ii), the induced map $B_n(X,G,Y)\rightarrow B_n(X',G',Y')$ is a homotopy pushout along the horizontal direction of the diagram
\[
\xymatrix{
	B_{n-1}(X,G,Y) \ar[d]_-{B_{n-1}(\phi,f,\psi)} &
		\partial\mathcal{Q}_n\times X\times G^{\times n}\times Y \ar[l] \ar[r] \ar[d]^-{\operatorname{id}\times\phi_0\times f_1^{\times n}\times\psi_0} &
		\mathcal{Q}_n\times X\times G^{\times n}\times Y \ar[d]^-{\operatorname{id}\times\phi_0\times f_1^{\times n}\times\psi_0} \\
	B_{n-1}(X',G',Y') &
		\partial\mathcal{Q}_n\times X'\times G'^{\times n}\times Y' \ar[l] \ar[r] &
		\mathcal{Q}_n\times X'\times G'^{\times n}\times Y'	
}
\]
such that the left square is given the appropriate homotopy and the right square strictly commutes.
Then the assertion follows by induction on $n$.
\end{proof}

\begin{lem}
\label{lem_B_quasifibration}
Let $G$ be a well-pointed topological monoid, $X$ a right $G$-space, and $Y$ a left $G$-space.
If the action $g\colon Y\to Y$ is a weak equivalence for any $g\in G$, then the projection $B_n(X,G,Y)\rightarrow B_n(X,G,*)$ is a quasifibration.
Moreover, if $G$ is a topological group, then the projection $B_n(X,G,Y)\rightarrow B_n(X,G,*)$ is a fiber bundle.
\end{lem}

\begin{proof}
The first half follows from applying the well-known criterion \cite[Theorem 2.7]{May90}.
For the proof of the latter half, see \cite[Theorem 8.2]{May75}.
\end{proof}

\begin{prp}
\label{prp_homomorphism_fiber_seq}
Let $f:G\rightarrow H$ be a homomorphism between well-pointed grouplike topological monoids.
Then the following sequence of maps is a homotopy fiber sequence
\begin{align*}
G\xrightarrow{f}H\rightarrow B(*,G,H)\rightarrow BG\xrightarrow{Bf}BH.
\end{align*}
\end{prp}

\begin{proof}
The sequence of the left three terms is equivalent to the sequence
\begin{align*}
G\rightarrow B(G,G,H)\rightarrow B(*,G,H)
\end{align*}
which is a homotopy fiber sequence by Lemma \ref{lem_B_quasifibration}.
Again by Lemma \ref{lem_B_quasifibration}, the sequence of the middle three terms is also a homotopy fiber sequence.
For the right middle terms, they also constitute a homotopy fiber sequence since the topological pullback square 
\begin{align*}
\xymatrix{
	B(*,G,H) \ar[r]^-{B(*,f,\mathrm{id}_H)} \ar[d]
		& B(*,H,H) \ar[d] \\
	B(*,G,*) \ar[r]^-{B(*,f,*)}
		& B(*,H,*)
}
\end{align*}
is a homotopy pullback as well by Lemma \ref{lem_B_quasifibration}.
\end{proof}

We note that there is a natural homeomorphism $B_1G\cong\Sigma G$.
By the exponential law
\begin{align*}
\Map_0(\Sigma G,\Sigma G)\cong\Map_0(G,\Omega\Sigma G),
\end{align*}
we have a natural map $E:G\rightarrow\Omega\Sigma G$.

\begin{lem}
\label{lem_G_to_OmegaBG}
For a grouplike topological monoid $G$, the composite
\begin{align*}
G\xrightarrow{E}\Omega\Sigma G\cong\Omega B_1G\xrightarrow{\Omega\iota_1}\Omega BG
\end{align*}
is a weak equivalence.
\end{lem}

\begin{proof}
Using the contracting homotopy of $B(\ast,G,G)$ in the proof of Lemma \ref{lem_B(X,G,G)}, we obtain the following commutative diagram:
\[
\xymatrix{
	G \ar[r] \ar[d] &
		EG \ar[r] \ar[d] &
		BG \ar@{=}[d] \\
	\Omega BG \ar[r] &
		PBG \ar[r] &
		BG.
}
\]
By Lemma \ref{lem_B_quasifibration}, the horizontal lines are homotopy fiber sequences.
Then, by the definition of the homotopy $\kappa_0$, the map $G\to\Omega BG$ in this diagram is equal to the above composite and hence is a weak equivalence.
\end{proof}

For $n<\infty$, we define a functor $D_n:\mathcal{A}_n^{\mathrm{R}}\to\mathbf{CG}$ by
\begin{align*}
	D_n(X,G)=(B_n(X,G,G)\cup(\mathcal{Q}_{n+1}\times X\times G^{\times (n+1)}))/\sim
\end{align*}
with the following identifications:
\begin{align*}
\begin{array}{cl}
\text{(i)} &
(\delta_k^0(\mathbf{t});x,g_1,\ldots,g_{n+1})\sim\left\{
	\begin{array}{ll}
		\lbrack\mathbf{t};xg_1,g_2,\ldots,g_{n+1},*\rbrack\in B_n(X,G,G) & \text{for }k=0 \\
		\lbrack\mathbf{t};x,g_1,\ldots,g_kg_{k+1},\ldots,g_{n+1},*\rbrack\in B_n(X,G,G) & \text{for }0<k<n \\
		\lbrack\mathbf{t};x,g_1,\ldots,g_{n+1}\rbrack\in B_n(X,G,G) & \text{for }k=n, \\
	\end{array}
	\right.
\\
\text{(ii)} &
(\mathbf{t};x,g_1,\ldots,g_{k-1},*,g_{k+1},\ldots,g_{n+1})\sim\lbrack\sigma_k(\mathbf{t});x,g_1,\ldots,g_{k-1},g_{k+1},\ldots,g_{n+1}\rbrack\in B_n(X,G,G).
\end{array}
\end{align*}
The induced maps are defined in the similar manner to $B_n$.
Then there are natural maps
\begin{align*}
&D_{n-1}(X,G)\subset B_n(X,G,G)\subset D_n(X,G), && \\
&D_n(X,G)\rightarrow B_{n+1}(X,G,*), && [\mathbf{t};x,g_1,\ldots,g_{n+1}]\mapsto[\mathbf{t};x,g_1,\ldots,g_{n+1},*],\\
&D_n(X,G)\rightarrow X, && [\mathbf{t};x,g_1,\ldots,g_{n+1}]\mapsto xg_1\cdots g_{n+1},
\end{align*}
for $(\mathbf{t};x,g_1,\ldots,g_{n+1})\in\mathcal{Q}_{n+1}\times X\times G^{\times (n+1)}$.
Like the proof of Lemma \ref{lem_B(X,G,G)}, one can see that the last map is a deformation retraction of the inclusion $X\subset D_n(X,G)$.
We denote $D_nG:=D_n(\ast,G)$.
The well-known homotopy equivalence $E_nG\simeq\Sigma^nG^{\wedge (n+1)}$ comes from the homeomorphism $E_nG/D_{n-1}G\cong\Sigma^nG^{\wedge (n+1)}$.
We also have a homeomorphism $D_nG/E_nG\cong\Sigma^{n+1}G^{\wedge (n+1)}$.
Then we obtain the well-known homotopy cofiber sequence
\[
	\Sigma^nG^{\wedge (n+1)}\to B_nG\xrightarrow{\iota_n^{n+1}}B_{n+1}G\to\Sigma^{n+1}G^{\wedge (n+1)}.
\]

\begin{rem}
The author could not find in the literature that the homotopy equivalence $E_nG\simeq\Sigma^nG^{\wedge (n+1)}$ can be obtained as above.
For example, Stasheff proved this homotopy equivalence by induction on $n$ and some homology Mayer--Vietoris sequence in \cite{Sta63a}.
\end{rem}

Let us consider a map $(D_{n-1}G,E_{n-1}G)\rightarrow(D_{n-1}G\vee\Sigma^nG^{\wedge n},E_{n-1}G)$ defined as follows:
for an element $\lbrack\gamma_{\mathbf{i}}^1(\mathbf{t};\mathbf{s}_1,\ldots,\mathbf{s}_r);*,g_1,\ldots,g_i\rbrack\in E_{n-1}G$,
\begin{align*}
\lbrack\gamma_{\mathbf{i}}^1(\mathbf{t};\mathbf{s}_1,\ldots,\mathbf{s}_r);*,g_1,\ldots,g_i\rbrack\mapsto\lbrack\mathbf{t};*,g_{i_1+1}\cdots g_{i_1+i_2},\ldots,g_{i_1+\cdots+i_{r-1}+1}\cdots g_i\rbrack\in E_{n-1}G,
\end{align*}
and for an element $(\delta_{n+1}^{t_{n+1}}(\gamma_{\mathbf{i}}^1(\mathbf{t};\mathbf{s}_1,\ldots,\mathbf{s}_r));*,g_1,\ldots,g_n)\in\mathcal{Q}_n\times *\times G^{\times n}$,
\begin{align*}
&\lbrack\delta_{n+1}^{t_{n+1}}(\gamma_{\mathbf{i}}^1(\mathbf{t};\mathbf{s}_1,\ldots,\mathbf{s}_r));*,g_1,\ldots,g_n\rbrack \\
&\mapsto\left\{
\begin{array}{ll}
\lbrack\mathbf{s}_1;g_1,\ldots,g_n\rbrack\in\Sigma^nG^{\wedge n} & (r=1)\\
\lbrack\delta_{n+1}^{t_{n+1}}(\mathbf{t});*,g_{i_1+1}\cdots g_{i_1+i_2},\ldots,g_{i_1+\cdots+i_{r-2}+1}\cdots g_{i_1+\cdots+i_{r-1}},*\rbrack\in E_{n-1}G & (1<r<n) \\
\lbrack\delta_{n+1}^{t_{n+1}}(\mathbf{t});*,g_1,\ldots,g_n\rbrack\in D_{n-1}G & (r=n),
\end{array}
\right.
\end{align*}
where $\Sigma^nG^{\wedge n}$ is considered as $\Sigma^nG^{\wedge n}=[0,1]^{\times n}\times G^{\times n}/(\partial[0,1]^{\times n}\times G^{\times n}\cup[0,1]^{\times n}\times T_nG)$.
This map is depicted in Figure \ref{fig_Dn}.
We call this map the \textit{pinch map}.
The pinch map restricted to $E_{n-1}G$ is naturally homotopic to the identity.
In the following lemma, we denote the reduced cone of a pointed space $X$ by $CX$.

\begin{figure}
\centering
\includegraphics[width=8cm,clip]{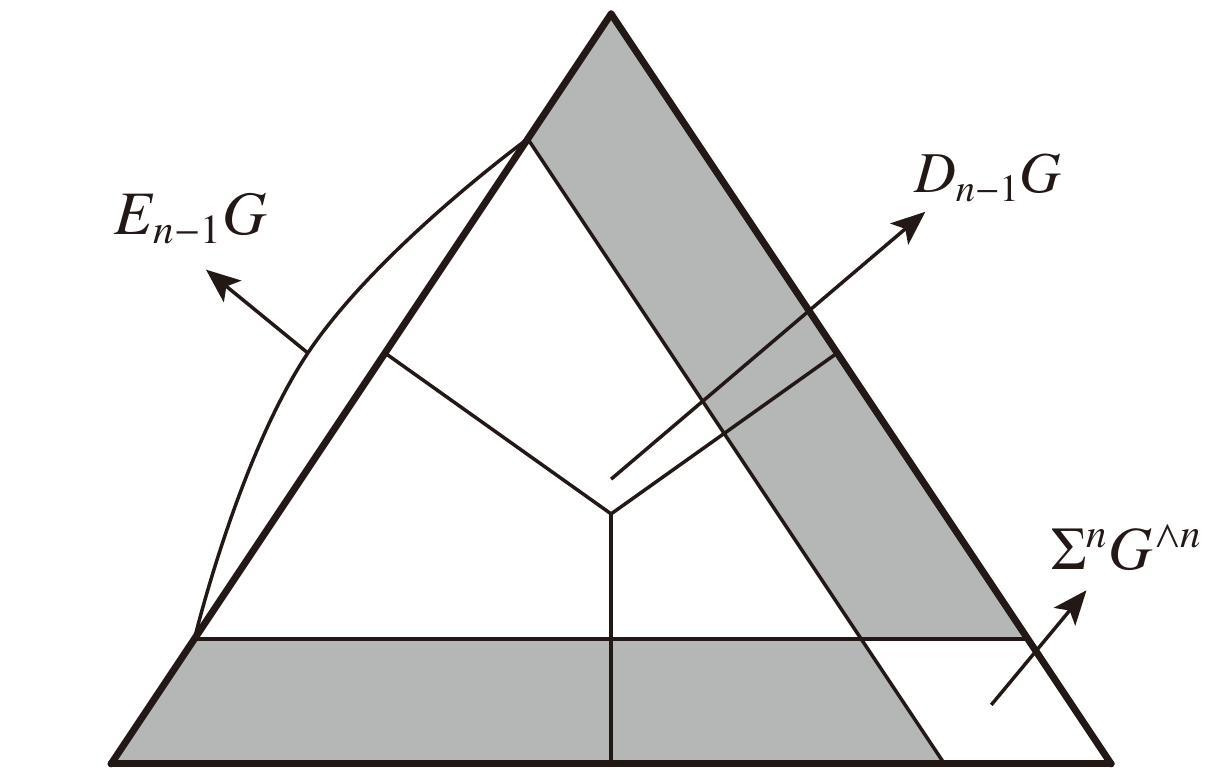}
\caption{}
\label{fig_Dn}
\end{figure}

\begin{lem}
\label{lem_pinch_map}
Let $G$ be a well-pointed topological monoid and $X$ a right $G$-space.
\begin{enumerate}[(i)]
\item
Then the pushout of the diagram
\begin{align*}
D_{n-1}(X,G)\leftarrow B_{n-1}(X,G,G)\rightarrow B_{n-1}(X,G,*)
\end{align*}
is naturally homeomorphic to $B_n(X,G,*)$ by the map induced from the map $D_{n-1}(X,G)\rightarrow B_n(X,G,*)$ as above and the inclusion $\iota_{n-1}^n\colon B_{n-1}(X,G,*)\rightarrow B_n(X,G,*)$.
\item
There exists a natural map $(CE_{n-1}G,E_{n-1}G)\rightarrow(D_{n-1}G,E_{n-1}G)$ that restricts to the identity on $E_{n-1}G$.
This map induces a diagram
\[
\xymatrix{
	CE_{n-1}G \ar[r] \ar[d] &
		CE_{n-1}G\vee\Sigma^nG^{\wedge n} \ar[d] \\
	D_{n-1}G \ar[r] &
		D_{n-1}G\vee\Sigma^nG^{\wedge n}
}
\]
commutative up to homotopy that restricts to the natural homotopy on $E_{n-1}G$, where the horizontal arrows represent the pinch maps and the right vertical map is the identity on $\Sigma^nG^{\wedge n}$.
\end{enumerate}
\end{lem}

\begin{proof}
It is not difficult to observe that the naturally induced map from the pushout in (i) to $B_n(X,G,\ast)$ is a homeomorphism.
The inverse map is given by the pushout square as in the proof of Lemma \ref{lem_homotopy_invariance_of_B}.
Then the assertion (i) follows.

We regard the reduced cone $CE_{n-1}G$ as
\[
	CE_{n-1}G=\mathcal{Q}_1\times E_{n-1}G/(\mathcal{Q}_1\times\ast\cup\delta_1^0\times E_{n-1}G).
\]
Then the map $\kappa_{n-1}$ in the proof of Lemma \ref{lem_B(X,G,G)} induces a natural map $(CE_{n-1}G,E_{n-1}G)\rightarrow(D_{n-1}G,E_{n-1}G)$.
For this map, it is straightforward to see the desired homotopy commutativity of the diagram in (ii).
\end{proof}

For a space $X$, the space $\mathcal{W}X:=|\mathrm{Sing}(X)|$ defined to be the realization of the simplicial complex of singular simplices of $X$ is naturally a CW complex.
There is a natural weak equivalence $\mathcal{W}X\rightarrow X$.
In particular, if $G$ is a topological monoid or group, then $\mathcal{W}G$ is a CW complex equipped with the natural structure of a topological monoid or group, respectively, which is given by a cellular map $\mathcal{W}G\times\mathcal{W}G\rightarrow\mathcal{W}G$.
Moreover, the natural weak equivalence $\mathcal{W}G\rightarrow G$ is a homomorphism.
If $G$ is a topological group, so is $\mathcal{W}G$.
If $X$ is a right $G$-space, then $\mathcal{W}X$ is a right $\mathcal{W}G$-space with the cellular action $\mathcal{W}X\times\mathcal{W}G\rightarrow\mathcal{W}X$ and the natural maps $\mathcal{W}X\rightarrow X$ and $\mathcal{W}G\rightarrow G$ preserve the action.

\begin{lem}
\label{lem_B_htype_CW}
Let $G$ be a well-pointed topological monoid, $X$ a right $G$-space, and $Y$ a left $G$-space, all of which have the homotopy types of CW complexes.
Then $B_n(X,G,Y)$ has the homotopy type of a CW complex.
In particular, the natural map $B_n(\mathcal{W}X,\mathcal{W}G,\mathcal{W}Y)\rightarrow B_n(X,G,Y)$ is a homotopy equivalence.
\end{lem}

\begin{proof}
Using the homotopy invariance of homotopy pushout, this lemma is similarly proved as in Lemma \ref{lem_homotopy_invariance_of_B}.
\end{proof}

\section{Mapping spaces from projective spaces}
\label{section_mapping_from_projective}

Now, let us prove our main theorem.

\begin{thm}
\label{mainthm}
Let $G$ be a well-pointed topological monoid having the homotopy type of a CW complex and $G'$ a well-pointed grouplike topological monoid.
Then the following composite is a weak equivalence.
\begin{align*}
	\mathcal{A}_n(G,G')\xrightarrow{B_n}\Map_0(B_nG,B_nG')\xrightarrow{(\iota_n)_{\#}}\Map_0(B_nG,BG').
\end{align*}
\end{thm}

\begin{proof}
We prove this theorem by induction.
For $n=1$, the composite
\begin{align*}
\Map_0(G,G')\xrightarrow{\Sigma}
	\Map_0(\Sigma G,\Sigma G')\xrightarrow{(\iota_1)_{\#}}
	\Map_0(\Sigma G,BG')\cong
	\Map_0(G,\Omega BG')
\end{align*}
is equal to $\zeta_{\#}$, where $\zeta:G'\rightarrow\Omega BG'$ is as in Lemma \ref{lem_G_to_OmegaBG}.
Then, the composite $(\iota_1)_{\#}\circ B_1$ is a weak equivalence.

Suppose that the composite $(\iota_{n-1})_{\#}\circ B_{n-1}$ is a weak equivalence.
Consider the following homotopy commutative diagram of homotopy fiber sequences (see Section \ref{section_category_of_A_n-maps} for the definition of $\mathcal{A}_n^1$):
\begin{align*}
	\xymatrix{
		\Map_0(\Sigma^{n-1}G^{\wedge n},G') \ar[r] \ar[d]^{(\iota_1)_{\#}\circ\Sigma} & \mathcal{A}_n^1(G,G') \ar[r] \ar[d] & \mathcal{A}_{n-1}^1(G,G') \ar[d]^{\simeq} \\
		\Map_0(\Sigma^{n}G^{\wedge n},BG') \ar[r] & \Map_0(B_nG,BG') \ar[r]^-{(\iota_{n-1}^n)^{\#}} & \Map_0(B_{n-1}G,BG'),
	}
\end{align*}
where the homotopy fibers and maps between them are determined by the observation in Section \ref{section_category_of_A_n-maps} and Lemma \ref{lem_pinch_map}.
Moreover, one can see that the vertical maps preserve principal actions up to homotopy.
Since the central vertical arrow induces an surjection on $\pi_0$ by Proposition \ref{prp_mainthm}, which will be proved later, then it is a weak equivalence.

For $n=\infty$, the result follows since the projections $\mathcal{A}_n^1(G,G')\rightarrow\mathcal{A}_{n-1}^1(G,G')$ and $\Map_0(B_nG,B_nG')\rightarrow\Map_0(B_nG,BG')$ have the homotopy lifting property, and $\mathcal{A}_\infty^1(G,G')$ and $\Map_0(BG,BG')$ are the limits along the sequences of these projections, respectively.
\end{proof}

To complete the proof, all we have to show is that the composite $(\iota_n)_{\#}\circ B_n$ induces the surjection on $\pi_0$ for $n<\infty$.
First, we consider the case when $G'$ is a topological group.

\begin{lem}
\label{lem_mainthm}
Under the setting in Theorem \ref{mainthm}, suppose that $G'$ is a topological group and $n<\infty$.
Then the composite $(\iota_n)_{\#}\circ B_n$ induces the surjection on $\pi_0$.
\end{lem}

\begin{proof}
We show this by induction on $n$.
Let $F\colon B_nG\to BG'$ be a pointed map.
Suppose that there exists an $A_{n-1}$-map $f=(f,\{f_i\}_{i=1}^{n-1},\ell)$ such that $F$ restricts to $\iota_{n-1}\circ B_{n-1}f$ on $B_{n-1}G$.
Since $EG'\to BG'$ is a fiber bundle by Lemma \ref{lem_B_quasifibration}, there exists a dotted arrow $h$ in the diagram
\[
\xymatrix{
	D_{n-2}G \ar[r]^-{D_{n-2}f} \ar[d] &
		EG' \ar[d] \\
	D_{n-1}G \ar[r] \ar@{.>}[ur]^h &
		BG'
}
\]
such that this diagram commutes, where the arrow $D_{n-1}G\to BG'$ is the composite $D_{n-1}G\to B_nG\xrightarrow{F}BG'$.
Then one can extend the $A_{n-1}$-form of $f$ to the $A_n$-form $\{f_i\}_{i=1}^n$ uniquely such that the restriction of $h$ to $E_{n-1}G\to E_{n-1}G'$ coincides with $E_{n-1}f$.
Note that $h\colon D_{n-1}G\to EG'$ is homotopic to the composite $D_{n-1}G\xrightarrow{D_{n-1}f}D_{n-1}G'\to EG'$ rel $E_{n-1}G$.
Therefore, by (i) of Lemma \ref{lem_pinch_map}, $F$ is homotopic to $B_nf$.
\end{proof}

\begin{prp}
\label{prp_mainthm}
Under the setting in Theorem \ref{mainthm}, the composite $(\iota_n)_{\#}\circ B_n$ induces the surjection on $\pi_0$ for general $G'$ and $n<\infty$.
\end{prp}

\begin{proof}
Considering the weak equivalences $\mathcal{W}G'\to G'$ and $B\mathcal{W}G'\to BG'$, we may assume that $G'$ has the homotopy type of a CW complex by Proposition \ref{prp_equiv_mappingsp_A_n}.
By the results on simplicial homotopy theory (for example, see \cite{May67}), we can find a topological group $\tilde G$ which is a CW complex such that the classifying space $B\tilde{G}$ is homotopy equivalent to $BG'$, where we use the fact that $BG'$ has the homotopy type of a CW complex by Lemma \ref{lem_B_htype_CW}.
Then, by Theorem \ref{mainthm} for $G'$ and $\tilde G$ and Corollary \ref{cor_inverse_of_A_n-equivalence}, there exists an $A_\infty$-equivalence $G'\to \tilde G$.
Combining Proposition \ref{prp_equiv_mappingsp_A_n} and Lemma \ref{lem_homotopy_invariance_of_B} and \ref{lem_mainthm}, we obtain the desired surjectivity.
\end{proof}

This completes the proof of Theorem \ref{mainthm}.

Let $X$ be a path-connected well-pointed space of the homotopy type of a CW complex.
Then one can show that the Moore based loop space $\Omega^{\mathrm{M}}X$ is a well-pointed grouplike topological monoid having the homotopy type of a CW complex.

\begin{rem}
Let $X$ be a well-pointed space of the homotopy type of a CW complex.
Then there is a natural homotopy equivalence into the path-component containing the basepoint
\begin{align*}
	B\Omega^{\mathrm{M}}X=B(*,\Omega^{\mathrm{M}}X,*)\xleftarrow{\simeq}B(P^{\mathrm{M}}X,\Omega^{\mathrm{M}}X,*)\rightarrow X,
\end{align*}
where the left arrow is induced by the map $P^{\mathrm{M}}X\rightarrow *$ and the right arrow is induced by the evaluation $e:P^{\mathrm{M}}X\rightarrow X$.
This homotopy equivalence is checked by the similar argument to the proof of the following corollary.

Through this homotopy equivalence, Theorem \ref{mainthm} is recognized as the adjunction
\begin{align*}
	\mathcal{A}_n(G,\Omega^{\mathrm{M}}X)\simeq\Map_0(B_nG,X)
\end{align*}
in certain sense.
We will call the correspondence of the homotopy classes through this weak equivalence or that of Theorem \ref{mainthm} as the \textit{adjoint}.
With respect to this adjunction, we consider the unit map as in the next corollary.
\end{rem}

\begin{rem}
Let $G$ be a well-pointed topological monoid and $X$ a pointed space.
Norio Iwase pointed out to the author that the weak equivalence
\[
	\mathcal{A}_n(G,\Omega^{\mathrm{M}}X)\simeq\Map_0(B_nG,X)
\]
stated above is in fact a \textit{homotopy equivalence}.
The inverse map of this equivalence is given as in the Stasheff's lifting-extension argument in the proof of \cite[Theorem 4.5]{Sta63b}, which can be done continuously.
\end{rem}

\begin{rem}
As in \cite{Sta63a} and \cite{IM89}, projective spaces and $A_n$-maps are defined for $A_n$-spaces as well.
Then it is natural to ask whether Theorem \ref{mainthm} can be generalized for an $A_n$-space $G$ and an $A_\infty$-space $G'$ or not.
This might be carried out but needs many preparations about $A_n$-spaces.
This problem will be postponed for now.
\end{rem}

\begin{cor}
\label{cor_OmegaB_n}
Let $G$ be a well-pointed topological monoid having the homotopy type of a CW complex.
Then there exists an $A_n$-map $\eta:G\rightarrow\Omega^{\mathrm{M}}B_nG$ such that the adjoint $\iota_n\circ B_n\eta:B_nG\rightarrow B\Omega^{\mathrm{M}}B_nG$ is a homotopy equivalence and there is a homotopy commutative diagram
\begin{align*}
\xymatrix{
B_nG \ar[rr]^-{\iota_n\circ B_n\eta} \ar@{=}[d]
	&
	& B\Omega^{\mathrm{M}}B_nG \ar@{=}[d] \\
B_nG 
	& B(P^{\mathrm{M}}B_nG,\Omega^{\mathrm{M}}B_nG,*) \ar[l]_-{e_*} \ar[r]^-{q}
	& B\Omega^{\mathrm{M}}B_nG,
}
\end{align*}
where the map $e_*$ is induced from the evaluation $e:P^{\mathrm{M}}B_nG\rightarrow B_nG$ and $q$ from the map $P^{\mathrm{M}}B_nG\rightarrow *$.
Moreover, the composite
\begin{align*}
\mathcal{A}_\infty(\Omega^{\mathrm{M}}B_nG,G')\rightarrow\mathcal{A}_n(\Omega^{\mathrm{M}}B_nG,G')\xrightarrow{\eta^{\#}}\mathcal{A}_n(G,G')
\end{align*}
is a weak equivalence for any grouplike topological monoid $G'$.
\end{cor}

\begin{proof}
The first half follows from the following commutative diagram and Theorem \ref{mainthm}:
\begin{align*}
\xymatrix{
	\Omega^{\mathrm{M}}B_nG \ar[d] 
		& \Omega^{\mathrm{M}}B_nG \ar[d] \ar@{=}[l] \ar@{=}[r]
		& \Omega^{\mathrm{M}}B_nG \ar[d] \\
	P^{\mathrm{M}}B_nG \ar[d]
		& B(P^{\mathrm{M}}B_nG,\Omega^{\mathrm{M}}B_nG,\Omega^{\mathrm{M}}B_nG) \ar[d] \ar[l] \ar[r]
		& B(*,\Omega^{\mathrm{M}}B_nG,\Omega^{\mathrm{M}}B_nG) \ar[d] \\
	B_nG
		& B(P^{\mathrm{M}}B_nG,\Omega^{\mathrm{M}}B_nG,*) \ar[l] \ar[r]
		& B(*,\Omega^{\mathrm{M}}B_nG,*).
}
\end{align*}
The latter half can be checked by the following commutative diagram:
\begin{align*}
\xymatrix{
	\mathcal{A}_\infty(\Omega^{\mathrm{M}}B_nG,G') \ar[r]
		& \mathcal{A}_n(\Omega^{\mathrm{M}}B_nG,G') \ar[r]^-{\eta^{\#}}
		& \mathcal{A}_n(G,G') \\
	\mathcal{A}_\infty(\Omega^{\mathrm{M}}B_nG,\mathcal{W}G') \ar[r] \ar[u]_-{\simeq} \ar[d]^-{\simeq}
		& \mathcal{A}_n(\Omega^{\mathrm{M}}B_nG,\mathcal{W}G') \ar[r]^-{\eta^{\#}} \ar[u]_-{\simeq} \ar[d]^-{\simeq}
		& \mathcal{A}_n(G,\mathcal{W}G') \ar[u]_-{\simeq} \ar[d]^-{\simeq} \\
	\Map_0(B\Omega^{\mathrm{M}}B_nG,B\mathcal{W}G') \ar[r]^-{(\iota_n)^{\#}}
		& \Map_0(B_n\Omega^{\mathrm{M}}B_nG,B\mathcal{W}G') \ar[r]^-{(B_n\eta)^{\#}}
		& \Map_0(B_nG,B\mathcal{W}G'),
}
\end{align*}
where the composite of the arrows in the bottom row is a weak equivalence.
\end{proof}

\begin{rem}
The $A_n$-map $\eta\colon G\rightarrow\Omega^{\mathrm{M}}B_nG$ has been studied by Stasheff in \cite{Sta70}.
Later, C. A. McGibbon \cite{McG82} proved that $\eta$ is never an $A_{n+1}$-map for any connected non-contractible CW complex $G$.
\end{rem}

\section{Application: Evaluation fiber sequences}
\label{section_evaluation}
For simplicity, we discuss only about topological groups rather than general topological monoids.
But, using the technique of simplicial homotopy theory as in the proof of Proposition \ref{prp_mainthm}, our result may admit some generalization.

For the fundamental facts on the space of bundle maps, see Gottlieb's paper \cite{Got72}.
Let $G$ be a well-pointed topological group, $B$ a well-pointed space of the homotopy type of a CW complex and $P$ a principal $G$-bundle over $B$ classified by $\epsilon\colon B\to BG$.
The \textit{gauge group} $\mathcal{G}(P)$ of $P$ is the topological group consisting of the $G$-equivariant self maps on $P$ that induces the identity on the quotient $P/G\cong B$.
Denote the space of $G$-equivariant maps $P\rightarrow EG$ by $\mathcal{E}(P,EG)$.
The gauge group $\mathcal{G}(P)$ acts on $\mathcal{E}(P,EG)$ from the right by composition.
Then there is a Serre fibration $\beta:\mathcal{E}(P,EG)\rightarrow\Map(B,BG;\epsilon)$ which assigns the induced map on the quotient $P/G=B\rightarrow EG/G=BG$.
Moreover, it is known that $\beta$ is a principal $\mathcal{G}(P)$-fibration and $\pi_i(\mathcal{E}(P,EG))=0$ for all $i\ge 0$.
Let $\rho:\mathcal{E}(P,EG)\rightarrow EG$ be the evaluation at the basepoint.
This map is equivariant through the homomorphism $\rho:\mathcal{G}(P)\rightarrow\mathcal{G}(G)\cong G$ defined by the evaluation at the basepoint.
Define the subspace $\mathcal{E}_0(P,EG):=\rho^{-1}(*)\subset\mathcal{E}(P,EG)$ and the closed subgroup $\mathcal{G}_0(P):=\rho^{-1}(*)\subset\mathcal{G}(P)$.
One can check that $\beta:\mathcal{E}_0(P,EG)\rightarrow\Map_0(B,BG;\epsilon)$ is a Serre fibration with a principal action by $\mathcal{G}_0(P)$ and $\pi_i(\mathcal{E}_0(P,EG))=0$ for all $i>0$.

For $g\in G$, the conjugation and the left translation
\begin{align*}
&\alpha_g:G\rightarrow G,&&\alpha_g(x)=gxg^{-1},\\
&L_g:G\rightarrow G, &&L_g(x)=gx
\end{align*}
induce a $G$-equivariant map $E\alpha_g:=B(*,\alpha_g,L_g):EG\rightarrow EG$ and a map $B\alpha_g:BG\rightarrow BG$.
These maps satisfy $E\alpha_{gg'}=E\alpha_g\circ E\alpha_{g'}$ and $B\alpha_{gg'}=B\alpha_g\circ B\alpha_{g'}$, and the following diagram commutes:
\begin{align*}
\xymatrix{
	EG \ar[r]^-{E\alpha_g} \ar[d]
		& EG \ar[d] \\
	BG \ar[r]^-{B\alpha_g}
		& BG.
}
\end{align*}

From now on, we use the notation $\mathcal{E}:=\mathcal{E}(P,EG)$, $\mathcal{E}_0:=\mathcal{E}_0(P,EG)$, $\mathcal{G}:=\mathcal{G}(P)$, $\mathcal{G}_0:=\mathcal{G}_0(P)$, $M:=\Map(B,BG;\epsilon)$ and $M_0:=\Map_0(B,BG;\epsilon)$ for simplicity of diagrams.
Following \cite[Section 6]{KK10}, consider free actions
\begin{align*}
&\mathcal{G}\times(EG\times\mathcal{E}_0)\rightarrow EG\times\mathcal{E}_0,
	&& (\varphi,(u,f))\mapsto(u\rho(\varphi)^{-1},E\alpha_{\rho(\varphi)}\circ f\circ\varphi^{-1}), \\
&G\times (EG\times M_0)\rightarrow EG\times M_0,
	&& (g,(F,u))\mapsto(ug^{-1},B\alpha_g\circ F).
\end{align*}
Here, the induced map from $\mathcal{G}$ into each fiber of the Serre fibration $\mathcal{E}_0\times EG\to M_0\times_GEG$ is a homeomorphism.

Let us consider the commutative diagram
\[
\xymatrix{
	\mathcal{G} \ar[r] &
		EG\times\mathcal{E}_0 \ar[r] &
		EG\times_GM_0 \\
	\mathcal{WG} \ar[u] \ar@{=}[d] \ar[r] &
		B(\mathcal{W}EG\times\mathcal{WE}_0,\mathcal{WG},\mathcal{WG}) \ar[u] \ar[d] \ar[r] &
		B(\mathcal{W}EG\times\mathcal{WE}_0,\mathcal{WG},\ast) \ar[u] \ar[d] \\
	\mathcal{WG} \ar[r] &
		B(\ast,\mathcal{WG},\mathcal{WG}) \ar[r] &
		B(\ast,\mathcal{WG},\ast) \\
	\mathcal{WG} \ar[d] \ar@{=}[u] \ar[r] &
		B(\mathcal{WE},\mathcal{WG},\mathcal{WG}) \ar[d] \ar[u] \ar[r] &
		B(\mathcal{WE},\mathcal{WG},\ast) \ar[d] \ar[u] \\
	\mathcal{G} \ar[r] &
		\mathcal{E} \ar[r] &
		M.
}
\]
Then each row is a Serre fibration and each vertical arrow is a weak equivalence.
By the similar argument, we obtain the following commutative diagram:
\[
\xymatrix{
	& EG\times M_0 \ar[r] \ar `l_[l]`[d]`[dddd][dddd] &
		EG\times_GM_0 \ar[r] &
		BG \\
	& B(\mathcal{W}EG\times\mathcal{WE}_0,\mathcal{WG}_0,\ast) \ar[u] \ar[d] \ar[r] \ar `l[d]`[dd][dd] &
		B(\mathcal{W}EG\times\mathcal{WE}_0,\mathcal{WG},\ast) \ar[u] \ar[d] \ar[r] &
		B(\mathcal{W}EG,\mathcal{W}G,\ast) \ar[u] \ar[d] \\
	& B(\ast,\mathcal{WG}_0,\ast) \ar[r] &
		B(\ast,\mathcal{WG},\ast) \ar[r] &
		B(\ast,\mathcal{W}G,\ast) \\
	& B(\mathcal{WE}_0,\mathcal{WG}_0,\ast) \ar[u] \ar[d] \ar[r] &
		B(\mathcal{WE},\mathcal{WG},\ast) \ar[u] \ar[d] \ar[r] &
		B(\mathcal{W}EG,\mathcal{W}G,\ast) \ar[u] \ar[d] \\
	& M_0 \ar[r] &
		M \ar[r] &
		BG,
}
\]
where all the vertical arrows are weak equivalences.
From this diagram, we obtain the following.

\begin{lem}
\label{lem_evaluation_Borel}
Let $G$ be a well-pointed topological group, $B$ a well-pointed space of the homotopy type of a CW complex, and $P$ a principal $G$-bundle over $B$.
Then there exists a CW complex $X,X'$ and the following homotopy commutative diagram
\begin{align*}
\xymatrix{
	\Map_0(B,BG;\epsilon) \ar[r] \ar`l[d]`[dd]_{\operatorname{id}}[dd] &
		EG\times_G\Map_0(B,BG;\epsilon) \ar[r] &
		BG \\
	X \ar[u]^-{\simeq} \ar[d]_-{\simeq} \ar[r] &
		X' \ar[u]^-{\simeq} \ar[d]_-{\simeq} \ar[r] &
		BG \ar@{=}[u] \ar@{=}[d] \\
	\Map_0(B,BG;\epsilon) \ar[r] &
		\Map(B,BG;\epsilon) \ar[r] &
		BG,
}
\end{align*}
where the vertical arrows are weak equivalences.
\end{lem}

\begin{rem}
Intuitively, this result states that the top and bottom rows are equivalent as a fiber sequence.
But we cannot expect the existence of a direct weak equivalence between $EG\times_G\Map_0(B,BG;\epsilon)$ and $\Map(B,BG;\epsilon)$ in general.
\end{rem}

The conjugation defines a homomorphism
\begin{align*}
G\rightarrow\mathcal{A}_n(G,G),\quad g\mapsto\alpha_g
\end{align*}
and a left $G$-action
\begin{align*}
G\times\mathcal{A}_n(G,G)\rightarrow\mathcal{A}_n(G,G),\quad (g,f)\mapsto\alpha_g\circ f.
\end{align*}
On the other hand, we have a homomorphism
\begin{align*}
G\rightarrow\Map_0(BG,BG),\quad g\mapsto B\alpha_g
\end{align*}
and a left $G$-action
\begin{align*}
G\times\Map_0(B_nG,BG)\rightarrow\Map_0(B_nG,BG),\quad (g,F)\mapsto B\alpha_g\circ F.
\end{align*}

\begin{lem}
\label{lem_action_on_A_n}
For a well-pointed topological group $G$ of the homotopy type of a CW complex, the map
\begin{align*}
(\iota_n)_{\#}\circ B_n&:\mathcal{A}_n(G,G)\rightarrow\Map_0(B_nG,BG),
\end{align*}
is a $G$-equivariant weak equivalence with respect to the above left $G$-action.
\end{lem}

\begin{proof}
This immediately follows from Theorem \ref{mainthm}.
\end{proof}

Let us denote the subspace of weak $A_n$-equivalences by $\mathcal{A}_n(G,G;\mathrm{eq})\subset\mathcal{A}_n(G,G)$, whose basepoint is the identity $\mathrm{id}_G\in\mathcal{A}_n$.
A subspace $\overline{\Map}_0(B_nG,BG)\subset\Map_0(B_nG,BG)$ is defined as follows: for a map $F\in\Map_0(B_nG,BG)$, $F$ is contained in $\overline{\Map}_0(B_nG,BG)$ if and only if the adjoint $A_n$-map $G\to G$ is a weak $A_n$-equivalence.
We also denote the union of path-components in $\Map(B_nG,BG)$ that intersect nontrivially with $\overline{\Map}_0(B_nG,BG)$ by $\overline{\Map}(B_nG,BG)$.
The basepoint of $\overline{\Map}_0(B_nG,BG)$ and $\overline{\Map}(B_nG,BG)$ is the inclusion $\iota_n\colon B_nG\rightarrow BG$.

\begin{thm}
\label{thm_extension_evaluation}
Let $G$ be a well-pointed topological group having the homotopy type of a CW complex.
Then there is a homotopy fiber sequence
\begin{align*}
\xymatrix{
\overline{\Map}_0(B_nG,BG)\to\overline{\Map}(B_nG,BG)\to BG\to B\mathcal{W}\mathcal{A}_n(G,G;\mathrm{eq})
}
\end{align*}
such that the map $BG\rightarrow B\mathcal{W}\mathcal{A}_n(G,G;\mathrm{eq})$ is induced from the homomorphism $G\rightarrow\mathcal{A}_n(G,G;\mathrm{eq})$ giving the conjugation.
\end{thm}

\begin{proof}
By Proposition \ref{prp_homomorphism_fiber_seq}, the sequence
\[
	\mathcal{A}_n(G,G;\mathrm{eq})\to B(*,G,\mathcal{A}_n(G,G;\mathrm{eq}))\to BG\to B\mathcal{W}\mathcal{A}_n(G,G;\mathrm{eq})
\]
with respect to the conjugation action of $G$ on $\mathcal{A}_n(G,G;\mathrm{eq})$ is a homotopy fiber sequence.
Then, combining Lemma \ref{lem_evaluation_Borel} and \ref{lem_action_on_A_n}, we obtain the desired homotopy fiber sequence.
\end{proof}

\begin{rem}
As remarked in Section \ref{section_Introduction}, if $n=\infty$, the extension in Theorem \ref{thm_extension_evaluation} coincides with the well-known fiber sequence
\begin{align*}
G\rightarrow\overline{\Map}_0(BG,BG)\rightarrow\overline{\Map}(BG,BG)\rightarrow BG\rightarrow B\mathcal{W}\overline{\Map}_0(BG,BG)\rightarrow B\mathcal{W}\overline{\Map}(BG,BG),
\end{align*}
where the monoid structures on $\overline{\Map}_0(BG,BG)$ and $\overline{\Map}(BG,BG)$ are given by compositions.
\end{rem}

The next example shows that the extension in Theorem \ref{thm_extension_evaluation} is the maximum.

\begin{ex}
Kishimoto--Kono--Theriault \cite[Theorem 1.3]{KKT13} showed that $\Omega\Map(S^4,B\SU(2)_{(5)};\iota_1)$ is not homotopy commutative.
This implies that $\Map(S^4,B\SU(2);\iota_1)$ is never delooped.
\end{ex}

But this is not always the case.

\begin{ex}
Let $T$ be the $m$-dimensional compact torus.
Then the maps
\begin{align*}
	&\overline{\Map}_0(BT,BT)\to\overline{\Map}_0(B_nT,BT)\\
	&\overline{\Map}(BT,BT)\to\overline{\Map}(B_nT,BT)
\end{align*}
are weak equivalences for $n\ge 1$.
Moreover, the evaluation fiber sequence
\[
	\overline{\Map}_0(BT,BT)\to\overline{\Map}(BT,BT)\to BT\to B\mathcal{W}\overline{\Map}_0(BT,BT)\to B\mathcal{W}\overline{\Map}(BT,BT)
\]
is equivalent to the sequence
\[
	\GL(m,\mathbb{Z})\to BT\rtimes\GL(m,\mathbb{Z})\to BT\to B\GL(m,\mathbb{Z})\to B(BT\rtimes\GL(m,\mathbb{Z})).
\]
This no longer extends.
This can be seen by observing the action of $\pi_1(B(BT\rtimes\GL(m,\mathbb{Z})))$ on $\pi_3(B(BT\rtimes\GL(m,\mathbb{Z})))$.
\end{ex}

As a step to observe the non-extendability, we conjecture as follows.

\begin{conjecture}
For a non-commutative compact connected Lie group $G$, $\Map(B_nG,BG;\iota_n)$ is never delooped for $1\le n<\infty$.
\end{conjecture}

\begin{rem}
Related to this conjecture, an upper bound of the homotopy nilpotency of $\mathcal{G}(E_nG)\simeq\Omega\Map(B_nG,BG;\iota_n)$ was given by Crabb--Sutherland--Zhang \cite{CSZ99} for general Lie groups $G$.
But the author does not know any result implying the homotopy non-commutativity of $\mathcal{G}(E_nG)$.
\end{rem}

\section{Application: Higher homotopy commutativity}
\label{section_appl_comm}
There are several notions of higher homotopy commutativity.
Sugawara introduced the strong homotopy commutativity in \cite{Sug60}, which is naturally generalized to Sugawara $C^n$-spaces \cite{McG89}.
Williams introduced another higher homotopy commutativity called Williams $C_n$-spaces in \cite{Wil69}.
Hemmi and Kawamoto defined $C_k(n)$-spaces in \cite{HK11}.
Kishimoto and Kono also considered certain higher commutativity called $C(k,\ell)$-spaces in \cite{KK10}.
First we compare these commutativities.
Recall that they are described by using projective spaces as follows.

\begin{prp}
\label{prp_higher_comutativities}
Let $G$ be a well-pointed grouplike topological monoid having the homotopy type of a CW complex.
Then the following statements hold:
\begin{enumerate}[(i)]
\item
$G$ is a Williams $C_n$-space if and only if the map $(\iota_1,\ldots,\iota_1)\colon(\Sigma G)^{\vee n}\rightarrow BG$ extends over the product $(\Sigma G)^{\times n}$,

\item
$G$ is a $C(k,\ell)$-space if and only if the map $(\iota_k,\iota_\ell)\colon B_kG\vee B_\ell G\rightarrow BG$ extends over the product $B_kG\times B_\ell G$,

\item
$G$ is a $C_k(n)$-space if and only if the map $(\iota_k,\iota_n)\colon B_kG\vee B_nG\rightarrow BG$ extends over the union $\bigcup_{i+j=n,i\le k}B_iG\times B_jG$,

\item
$G$ is a Sugawara $C^n$-space if and only if $G$ is a $C_n(n)$-space.
\end{enumerate}
\end{prp}

See \cite{Sau95} for the proof of (i), \cite{KK10} for (ii), and \cite{HK11} for (iii) and (iv).
Obviously, any Sugawara $C^n$-space is a $C_k(n)$-space for $k\le n$, and any $C_k(n)$-space is a $C(k,n-k)$-space.
The Williams $C_n$-space is related with other homotopy commutativities as the following lemma.

\begin{lem}
Let $G$ be a well-pointed grouplike topological monoid having the  homotopy type of a CW complex.
If $G$ is a $C(k,\ell)$-space, then $G$ is a $C_{k+\ell}$-space.
\end{lem}

\begin{proof}
It is sufficient to prove that the map $j\colon(\iota_1,\ldots,\iota_1)\colon(\Sigma G)^{\vee (k+\ell)}\rightarrow BG$ can be extended over the product $(\Sigma G)^{\times(k+\ell)}$ by Proposition \ref{prp_higher_comutativities} (i).
Since $G$ is a $C(k,\ell)$-space, for connected CW complexes $A$ and $B$ such that $\cat A\le k$ and $\cat B\le\ell$, any map $A\vee B\to BG$ extends over the product $A\times B$.
Then the map $j$ can be extended since the $n$-fold product of the suspension spaces have the L--S category less than or equal to $n$.
\end{proof}

Now we let $G$ be a well-pointed topological group having the homotopy type of a CW complex.
In Section \ref{section_evaluation}, we saw that the connecting map $\delta\colon G\rightarrow\Map_0(B_nG,BG)$ in the evaluation fiber sequence
\begin{align*}
G\xrightarrow{\delta}\overline{\Map}_0(B_nG,BG)\rightarrow\overline{\Map}(B_nG,BG;\iota_n)\rightarrow BG
\end{align*}
is $\delta(g)=B\alpha_g\circ\iota_n$ and is identified with the homomorphism $G\rightarrow\mathcal{A}_n(G,G)$ induced by the conjugation.

\begin{thm}
\label{thm_C(k,l)_connecting}
Let $G$ be a well-pointed topological group having the pointed homotopy type of a CW complex.
Then $G$ is a $C(k,\ell)$-space if and only if the homomorphism $G\rightarrow\mathcal{A}_\ell(G,G)$ giving the conjugation is homotopic to the trivial map as an $A_k$-map.
\end{thm}

\begin{proof}
By Theorem \ref{mainthm}, the homomorphism $G\rightarrow\mathcal{A}_\ell(G,G)$ is homotopic to the trivial map as $A_k$-map if and only if the composite
\begin{align*}
B_kG\xrightarrow{\iota_k}BG\rightarrow B\mathcal{W}\mathcal{A}_n(G,G;\mathrm{eq})
\end{align*}
is null-homotopic.
By the evaluation fiber sequence in Theorem \ref{thm_extension_evaluation}, this condition is equivalent to the existence of the wedge sum $(\iota_k,\iota_\ell):B_kG\vee B_\ell G\rightarrow BG$ over the product $B_kG\times B_\ell G$.
By Proposition \ref{prp_higher_comutativities}, it is equivalent to $G$ being a $C(k,\ell)$-space.
\end{proof}

\begin{cor}
Let $G$ be a well-pointed topological group having the pointed homotopy type of a CW complex.
Then the following conditions are equivalent:
\begin{enumerate}[(i)]
\item
the classifying space $BG$ is an $H$-space,
\item
$G$ is a Sugawara $C^\infty$-space,
\item
the map $G\rightarrow\mathcal{A}_n(G,G)$ induced by the conjugation is homotopic as an $A_\infty$-map to the constant map to the identity.
\end{enumerate}
\end{cor}

For a pointed spaces $X$ and $Y$, we denote the half-smash product by
\begin{align*}
X\ltimes Y:=X\times Y/X\times *.
\end{align*}

\begin{cor}
Let $G$ be a well-pointed topological group having the homotopy type of a CW complex.
Then $G$ is a $C(1,\ell)$-space if and only if the map
\begin{align*}
G\ltimes B_\ell G\rightarrow BG,\quad (g,x)\mapsto B\alpha_g(\iota_\ell(x))
\end{align*}
is homotopic rel $B_\ell G$ to the composite of the projection $G\ltimes B_\ell G\rightarrow B_\ell G$ and the inclusion $\iota_\ell:B_\ell G\rightarrow BG$.
\end{cor}

\begin{rem}
\label{rem_connecting_Whitehead}
When $\ell=1$, $G\ltimes\Sigma G$ is naturally homotopy equivalent to $\Sigma(G\wedge G)\vee\Sigma G$.
Then it is easy to check that the composite $\Sigma(G\wedge G)\rightarrow G\ltimes\Sigma G\rightarrow BG$ is homotopic to the Whitehead product $[\iota_1,\iota_1]$.
This has been known by G. E. Lang \cite{Lan73}.
\end{rem}

\section{Application: $A_n$-types of gauge groups}
\label{section_appl_gauge}
The classification of $A_n$-types ($n\ge 2$) of gauge groups is first considered by M. C. Crabb and W. A. Sutherland for $n=2$ in \cite{CS00} after several works on the homotopy types begun with A. Kono's work \cite{Kon91}.
The author studied the $A_n$-types of the gauge groups of principal $\SU(2)$-bundles over $S^4$ for general $n$ in \cite{Tsu12} and \cite{Tsu15}.
In this section, we apply our result to the triviality of adjoint bundles.
It is an important problem in the classification of $A_n$-types of gauge groups.
The triviality we consider is defined as follows.

We refer to Section \ref{section_evaluation} for basic notions on gauge groups.
Let $G$ be a topological group and $P$ be a principal $G$-bundle over a space $B$.
The \textit{adjoint bundle} $\ad P$ is the associated bundle of $P$ induced by the conjugation on $G$ itself.
The adjoint bundle is naturally a fiberwise topological group, that is, there is a fiberwise multiplication $\ad P\times_B\ad P\rightarrow\ad P$ which makes the each fiber a topological group with continuous fiberwise inversion $\ad P\rightarrow\ad P$.
The space of sections $\Gamma(\ad P)$ of $\ad P$ is naturally isomorphic to the gauge group $\mathcal{G}(P)$.
Moreover, considering the obvious fiberwise version of $A_n$-map, each fiberwise $A_n$-map $\ad P\rightarrow\ad P$ induces an $A_n$-map $\mathcal{G}(P)\rightarrow\mathcal{G}(P')$.
If the underlying map of the fiberwise $A_n$-map is a homotopy equivalence, then the induced map on the gauge groups is also a homotopy equivalence.

\begin{dfn}
A fiberwise topological monoid $E\rightarrow B$ is said to be \textit{$A_n$-trivial} if there exist a topological monoid $G$ and a fiberwise $A_n$-map $B\times G\rightarrow E$ which restricts to a homotopy equivalence on each fiber.
\end{dfn}

Though the following proposition is partially proved in \cite{KK10}, we give another proof.

\begin{prp}
\label{prp_triviality_KK10}
Let $G$ be a well-pointed topological group and $B$ be a pointed space, both of which have the pointed homotopy type of CW complexes.
For a principal $G$-bundle $P$ over $B$ classified by $\epsilon:B\rightarrow BG$, the following conditions are equivalent:
\begin{enumerate}[(i)]
\item
$\ad P$ is $A_n$-trivial,

\item
the composite $B\xrightarrow{\epsilon}BG\rightarrow B\mathcal{W}\mathcal{A}_n(G,G;\mathrm{eq})$ is null-homotopic,

\item
the map $(\epsilon,\iota_n):B\vee B_nG\rightarrow BG$ extends over the product $B\times B_nG$.
\end{enumerate}
\end{prp}

\begin{proof}
By the evaluation fiber sequence in Theorem \ref{thm_extension_evaluation}
\begin{align*}
\overline{\Map}(B_nG,BG)\rightarrow BG\rightarrow B\mathcal{W}\mathcal{A}_n(G,G;\mathrm{eq}),
\end{align*}
the conditions (ii) and (iii) are equivalent.
Now we check the equivalence between (i) and (ii).
Consider the associated bundle $E=P\times_G\mathcal{A}_n(G,G;\mathrm{eq})$ induced from the homomorphism $G\rightarrow\mathcal{A}_n(G,G;\mathrm{eq})$ giving the conjugation.
By construction, the existence of a section of $E$ and that of a fiberwise $A_n$-map $B\times G\rightarrow\ad P$ which restricts to a homotopy equivalence on each fiber are equivalent.
The former is equivalent to the condition (ii).
The latter is equivalent to the condition (i).
This completes the proof.
\end{proof}

\begin{rem}
By an obstruction argument, this proposition extends to the classification theorem of fiberwise $A_n$-equivalence class.
For details, see \cite{Tsu12} and \cite{Tsu15}.
\end{rem}

\begin{prp}
Let $G$ be a well-pointed topological group and $B$ be a pointed space, both of which have the pointed homotopy types of CW complexes.
For a principal $G$-bundle $P$ over the suspension $\Sigma B$ classified by $\epsilon:\Sigma B\rightarrow BG$, the adjoint bundle $\ad P$ is trivial as a fiberwise $A_n$-space if and only if the map
\begin{align*}
B\ltimes B_nG\rightarrow BG,\quad (b,x)\mapsto B\alpha_{\epsilon'(b)}(\iota_n(x))
\end{align*}
is homotopic rel $B_nG$ to the composite of the projection $B\ltimes B_nG\rightarrow B_nG$ and the map $\iota_n:B_nG\rightarrow BG$, where $\epsilon':B\rightarrow G$ is the adjoint of $\epsilon$.
\end{prp}

\begin{proof}
This immediately follows from the evaluation fiber sequence
\begin{align*}
G\rightarrow\overline{\Map}_0(B_nG,BG)\rightarrow\overline{\Map}(B_nG,BG)\rightarrow BG\rightarrow B\mathcal{W}\mathcal{A}_n(G,G;\mathrm{eq})
\end{align*}
and Proposition \ref{prp_triviality_KK10}.
\end{proof}

\section{Application: $T_k^f$-spaces}
\label{section_appl_cyclic}
Iwase--Mimura--Oda--Yoon defined \textit{$C_k^f$-spaces} and \textit{$T_k^f$-spaces} in \cite{IMOY12}.
Since the terminology ``$C_k$'' is now confusing, we consider $T_k^f$-spaces.

\begin{dfn}
For a pointed map $f\colon A\rightarrow X$ between the well-pointed spaces of the homotopy types of CW complexes, $X$ is said to be a \textit{$T_k^f$-space} if there exists a map $f_k\colon A\times B_k\Omega^{\mathrm{M}}X\rightarrow X$ such that the following diagram commutes up to homotopy:
\begin{align*}
\xymatrix{
A\vee B_1\Omega^{\mathrm{M}}X \ar[dr]^{(f,\iota_1)} \ar[d]
	& \\
A\times B_k\Omega^{\mathrm{M}}X \ar[r]_-{f_k}
	& X.
}
\end{align*}
In particular, a $T_k^{\operatorname{id}_X}$-space is what J. Aguade \cite{Agu87} defined as a $T_k$-space.
\end{dfn}

\begin{prp}
\label{prp_C_k_T_k}
Let $A$ and $X$ be well-pointed spaces, which have the pointed homotopy types of CW complexes.
For a pointed map $f\colon A\rightarrow X$, $X$ is a $T_k^f$-space if and only if there exists a map $f_k\colon A\times B_k\Omega^{\mathrm{M}}X\rightarrow X$ such that the following diagram commutes up to homotopy:
\begin{align*}
\xymatrix{
A\vee B_k\Omega^{\mathrm{M}}X \ar[dr]^{(f,\iota_k)} \ar[d]
	& \\
A\times B_k\Omega^{\mathrm{M}}X \ar[r]_-{f_k}
	& X.
}
\end{align*}
\end{prp}

\begin{rem}
Iwase--Mimura--Oda--Yoon defined $C_k^f$-spaces by the condition in this proposition.
Therefore, it states that the $C_k^f$-space and the $T_k^f$-space are exactly the same concept.
\end{rem}

The if part of Proposition \ref{prp_C_k_T_k} is trivial.
The only if part immediately follows from the next lemma.

\begin{lem}
Let $G$ be a well-pointed grouplike topological monoid having the homotopy type of a CW complex.
For a pointed map $F\colon B_nG\rightarrow BG$, if the adjoint $G\rightarrow G$ of the composite $F\circ\iota_1^n:B_1G\rightarrow BG$ is a homotopy equivalence, then there exists a homotopy equivalence $F'\colon B_nG\rightarrow B_nG$ such that the composite $F\circ F':B_nG\rightarrow BG$ is pointed homotopic to $\iota_n$.
\end{lem}

\begin{proof}
By Theorem \ref{mainthm}, there exists an $A_n$-map $f\in\mathcal{A}_n(G,G)$ such that $\iota_n\circ B_nf$ is homotopic to $F$ and the underlying map of $f$ is a homotopy equivalence.
Then by Corollary \ref{cor_inverse_of_A_n-equivalence}, there exists an $A_n$-map $f'\in\mathcal{A}_n(G,G)$ which gives the inverse of $f$ in the homotopy category $\pi_0\mathcal{A}_n$.
Thus, for $F':=B_nf'$, we obtain that $F\circ F'$ is homotopic to $\iota_n$.
\end{proof}

There are certain relations between $T_k^f$-space and the objects in the preceding two sections as follows.
One can prove them by straightforward argument.

\begin{prp}[cf. \cite{HK11}]
For a well-pointed space $X$ of the homotopy type of a CW complex, $X$ is a $T_k$-space if and only if $\Omega^{\mathrm{M}}X$ is a $C(\infty,k)$-space.
\end{prp}

\begin{prp}
Let $G$ be a well-pointed topological group and $B$ be a well-pointed space, both of which have the homotopy types of CW complexes.
Then, for a principal $G$-bundle $P$ over $B$ classified by $\epsilon:B\rightarrow BG$, the adjoint bundle $\ad P$ is $A_n$-trivial if and only if $BG$ is a $T_n^{\epsilon}$-space.
\end{prp}

Let us denote the Moore free loop space of $X$ by $L^{\mathrm{M}}X$.
The next proposition is a generalization of Aguad\'e's definition of $T$-space in \cite{Agu87}.

\begin{prp}
Let $A$ and $X$ be well-pointed spaces of the homotopy types of CW complexes and $f\colon A\to X$ a pointed map.
Then the following two conditions are equivalent:
\begin{enumerate}[(i)]
\item
$X$ is a $T^f_k$-space,
\item
the pullback $f^\ast L^{\mathrm{M}}X$ is $A_k$-trivial.
\end{enumerate}
\end{prp}

\begin{proof}
Take a group model $G_X\xrightarrow{\simeq}\Omega^{\mathrm{M}}X$ which is an $A_\infty$-equivalence such that $G_X$ is a CW complex.
Then, by the following Lemma \ref{lem_adjoint_freeloop}, $f^\ast L^{\mathrm{M}}X$ is $A_k$-trivial if and only if $f^\ast\ad EG_X$ is $A_k$-trivial.
Combining this with Proposition \ref{prp_triviality_KK10}, we obtain the desired result.
\end{proof}

\begin{lem}
\label{lem_adjoint_freeloop}
Let $G$ be a well-pointed topological group having the homotopy type of CW complex.
Then, the adjoint bundle $\ad EG$ and the free loop space $L^{\mathrm{M}}BG$ are fiberwise $A_\infty$-equivalent as fiberwise topological monoids over $BG$.
\end{lem}

\begin{proof}
We just outline the proof.
As in \cite[Lemma 7.1]{Tsu12}, one can show that $\ad EG$ is fiberwise $A_\infty$-equivalent to $EG\times_G\Omega^{\mathrm{M}}BG$.
Using the map $EG\to PBG\subset P^{\mathrm{M}}BG$ in the proof of Lemma \ref{lem_G_to_OmegaBG}, one can construct a fiberwise $A_\infty$-equivalence
\[
	EG\times_G\Omega^{\mathrm{M}}BG\to L^{\mathrm{M}}BG.
\]
Composing these equivalences, we have the fiberwise $A_\infty$-equivalence $\ad EG\to L^{\mathrm{M}}BG$.
\end{proof}

We consider a natural family of subgroups of homotopy groups.

\begin{dfn}
For a well-pointed space $X$ of the homotopy type of a connected CW complex, define
\begin{align*}
G_{n,k}(X):=\{f\in\pi_n(X)\mid X\text{ is a }T_k^f\text{-space}\}.
\end{align*}
Equivalently, $G_{n,k}(X)$ is the image of the induced map of the evaluation map
\begin{align*}
\pi_n(\Map_0(B_k\Omega^{\mathrm{M}}X,X;\iota_k))\to\pi_n(X).
\end{align*}
\end{dfn}

For $k=\infty$, $G_n(X):=G_{n,\infty}(X)$ is the \textit{$n$-th Gottlieb group} introduced by D. H. Gottlieb \cite{Got69}.
If $k_1>k_2$, there is the inclusion $G_{n,k_1}(X)\subset G_{n,k_2}(X)$.
The space $X$ is said to be the \textit{Gottlieb space} if $G_{n,\infty}(X)=\pi_n(X)$ for any $n$.

Gottlieb \cite{Got69} also introduced the subgroup called the \textit{$n$-th Whitehead center}
\[
	P_n(X):=\{f\in\pi_n(X)\mid [f,g]=0\text{ for any }g\in\pi_*(X)\},
\]
where $[f,g]$ denotes the Whitehead product of $f$ and $g$.
By Remark \ref{rem_connecting_Whitehead}, $P_n(X)\supset G_{n,1}(X)$.
Then we have the sequence of inclusions
\[
	G_n(X)=G_{n,\infty}(X)\subset\cdots\subset G_{n,k}(X)\subset G_{n,k-1}(X)\subset\cdots\subset G_{n,1}(X)\subset P_n(X)\subset\pi_n(X).
\]

Now we interpret the result obtained by the author in \cite{Tsu15} in the language of $T^f_k$-spaces.

\begin{ex}
Let $p$ be an odd prime.
Denote the map $S^4\cong B_1\SU(2)\xrightarrow{\iota_1}B\SU(2)$ and their localizations by $\epsilon$.
The author proved in \cite{Tsu15} that $B\SU(2)_{(p)}$ is a $T_{\tfrac{(r+1)(p-1)}{2}-1}^{p^r\epsilon}$-space but not a $T_{\tfrac{(r+1)(p-1)}{2}}^{p^r\epsilon}$-space.
Then $G_{4,k}(B\SU(2)_{(p)})\subset\pi_4(B\SU(2)_{(p)})\cong\mathbb{Z}_{(p)}$ is the subgroup of index $p^r$ if $\tfrac{r(p-1)}{2}\le k<\tfrac{(r+1)(p-1)}{2}$ and $G_{4,\infty}(B\SU(2)_{(p)})=0$.

Each element $\alpha\in G_{4,\tfrac{r(p-1)}{2}}(B\SU(2)_{(p)})-G_{4,\tfrac{r(p-1)}{2}-1}(B\SU(2)_{(p)})$ has a non-trivial image in
\[
	\pi_3(\Map_0(B_{\tfrac{r(p-1)}{2}}\SU(2)_{(p)},B\SU(2)_{(p)};\iota_{\tfrac{r(p-1)}{2}}))
\]
but trivial in
\[
	\pi_3(\Map_0(B_{\tfrac{r(p-1)}{2}-1}\SU(2)_{(p)},B\SU(2)_{(p)};\iota_{\tfrac{r(p-1)}{2}-1})).
\]
Thus, considering the fiber sequence
\[
	\Map_0(\Sigma^kG^{\wedge k},BG;\iota_k)\to\Map_0(B_kG,BG;\iota_k)\to\Map_0(B_kG,BG;\iota_{k-1}),
\]
it lifts to a non-trivial element in
\[
	\pi_3(\Map_0(\Sigma^{\tfrac{r(p-1)}{2}}\SU(2)_{(p)}^{\wedge{\tfrac{r(p-1)}{2}}},B\SU(2)_{(p)};\iota_{\tfrac{r(p-1)}{2}}))\cong\pi_{2r(p-1)+2}(S^3_{(p)}).
\]
\end{ex}

\section{Application: Some remarks on homotopy pullback of $A_n$-maps}
\label{section_homotopypullback}

Let $G_1,G_2,G_3,G$ be well-pointed grouplike topological monoids of the homotopy types of CW complexes.
Consider the following homotopy commutative diagram in $\mathcal{A}_\infty$:
\[
\xymatrix{
	G \ar[r] \ar[d] &
		G_1 \ar[d] \\
	G_2 \ar[r] &
		G_3.
}
\]
Take another well-pointed topological monoid $H$ of the homotopy types of CW complex.
Then, for the induced two homotopy commutative diagrams
\[
\xymatrix{
	\mathcal{A}_n(H,G) \ar[r] \ar[d] &
		\mathcal{A}_n(H,G_1) \ar[d] \\
	\mathcal{A}_n(H,G_2) \ar[r] &
		\mathcal{A}_n(H,G_3),
}
\qquad
\xymatrix{
	\Map_0(B_nH,BG) \ar[r] \ar[d] &
		\Map_0(B_nH,BG_1) \ar[d] \\
	\Map_0(B_nH,BG_2) \ar[r] &
		\Map_0(B_nH,BG_3),
}
\]
the left square is a homotopy pullback if and only if so is the right by Theorem \ref{mainthm}.
In particular, considering the homotopy pullback of $A_n$-maps $H\to G_i$ for $i=1,2,3$ is equivalent to that of the induced maps $B_nH\to BG_i$ for $i=1,2,3$.
Moreover, if $n=\infty$ and the left square is pullback for any $H$, then $BG$ is the homotopy pullback of the diagram
\[
	BG_1\rightarrow BG_3\leftarrow BG_2.
\]


\begin{thebibliography}{99}

\bibitem[Agu87]{Agu87}
{J. Aguad\'e},
\textit{Decomposable free loop spaces},
{Canad. J. Math. \textbf{39} (1987), 938--955}.

\bibitem[BV73]{BV73}
{J. M. Boardman and R. M. Vogt},
{Homotopy invariant structures on topological spaces},
{Lecture Notes in Mathematics \textbf{347}, Springer-Verlag, Berlin, 1973}.

\bibitem[CS00]{CS00}
{M. C. Crabb and W. A. Sutherland},
\textit{Counting homotopy types of gauge groups},
{Proc. London Math. Soc. \textbf{81} (2000), 747--768}.

\bibitem[CSZ99]{CSZ99}
{M. C. Crabb, W. A. Sutherland and P. Zhang},
\textit{Homotopy nilpotency},
{Quart J. Math. Oxford \textbf{50} (1999), 179--196}.

\bibitem[DL59]{DL59}
{A. Dold and R. Lashof},
\textit{Principal quasifibrations and fibre homotopy equivalence of bundles},
{Illinois J. Math. \textbf{3} (1959), 285--305}.

\bibitem[Fuc65]{Fuc65}
{M. Fuchs},
\textit{Verallgemeinerte Homotopie-Homomorphismen und klassifizierende Ra\"ume},
{Math. Ann. \textbf{161} (1965), 197--230}.

\bibitem[Got69]{Got69}
{D. H. Gottlieb},
\textit{Evaluation subgroups of homotopy groups},
{Amer. J. Math. Soc. \textbf{91} (1969), 729--756}.

\bibitem[Got72]{Got72}
{D. H. Gottlieb},
\textit{Applications of bundle map theory},
{Trans. Amer. Math. Soc. \textbf{171} (1972), 23--50}.

\bibitem[Got73]{Got73}
{D. H. Gottlieb},
\textit{The total space of universal fibrations},
{Pacific J. Math. \textbf{46} (1973), 415--417}.

\bibitem[HK11]{HK11}
{Y. Hemmi and Y. Kawamoto},
\textit{Higher homotopy commutativity and the resultohedra},
{J. Math. Soc. Japan \textbf{63} (2011), 443--471}.

\bibitem[Hov99]{Hov99}
{M. Hovey},
{Model categories},
{Mathematical Surveys and Monographs, vol. 63, American Mathematical Society, Providence, RI, 1999}.

\bibitem[IM89]{IM89}
{N. Iwase and M. Mimura},
\textit{Higher homotopy associativity},
{Algebraic Topology, Proceedings of an International Conference held in Arcata, California, July 27 - August 2, 1986},
{Lecture Notes in Mathematics \textbf{1370}, 1989, 193--220}.

\bibitem[IMOY12]{IMOY12}
{N. Iwase, M. Mimura, N. Oda and S. Yoon},
\textit{The Milnor--Stasheff filtration on spaces and generalized cyclic maps},
{Canad. Math. Bull. \textbf{55} (2012), 523--536}.

\bibitem[KK10]{KK10}
{D. Kishimoto and A. Kono},
\textit{Splitting of gauge groups},
{Trans. Amer. Math. Soc. \textbf{362} (2010), 6715--6731}.

\bibitem[KKT13]{KKT13}
{D. Kishimoto, A.Kono and S. Theriault},
\textit{Homotopy commutativity in $p$-localized gauge groups},
{Proc. Roy. Soc. Edinburgh: Sect. A \textbf{143} (2013), 851--870}.

\bibitem[Kon91]{Kon91}
{A. Kono},
\textit{A note on the homotopy type of certain gauge groups},
{Proc. Roy. Soc. Edinburgh: Sect. A \textbf{117} (1991), 295--297}.

\bibitem[Lan73]{Lan73}
{G. E. Lang},
\textit{The evaluation map and EHP sequences},
{Pacific J. Math. \textbf{44} (1973), 201--210}.

\bibitem[May67]{May67}
{J. P. May},
{Simplicial Objects in Algebraic Topology},
{University of Chicago Press, (1967)}.

\bibitem[May75]{May75}
{J. P. May},
{Classifying spaces and fibrations},
{Mem. Amer. Math. Soc. \textbf{155} (1975)}.

\bibitem[May80]{May80}
{J. P. May},
{Fibrewise localization and completion},
{Trans. Amer. Math. Soc. \textbf{258} (1980), 127--146}.

\bibitem[May90]{May90}
{J. P. May},
\textit{Weak equivalences and quasifibrations},
{Groups of Self-Equivalences and Related Topics},
{Lecture Notes in Math. \textbf{1425} (1990), 91--101}

\bibitem[McG82]{McG82}
{C. A. McGibbon},
\textit{Multiplicative properties of power maps II},
{Trans. Amer. Math. Soc. \textbf{274} (1982), 479--508}.

\bibitem[McG89]{McG89}
{C. A. McGibbon},
\textit{Higher forms of homotopy commutativity and finite loop spaces},
{Math. Z. \textbf{201} (1989), 363--374}.

\bibitem[Sa\"u95]{Sau95}
{L. Sa\"umell},
\textit{Higher homotopy commutativity in localized groups},
{Math. Z. \textbf{219} (1995), 339-347}.

\bibitem[Sta63a]{Sta63a}
{J. D. Stasheff},
\textit{Homotopy associativity of H-spaces. I},
{Trans. Amer. Math. Soc. \textbf{108} (1963), 275--292}.

\bibitem[Sta63b]{Sta63b}
{J. D. Stasheff},
\textit{Homotopy associativity of H-spaces. II},
{Trans. Amer. Math. Soc. \textbf{108} (1963), 293--312}.

\bibitem[Sta70]{Sta70}
{J. Stasheff},
{H-spaces from a Homotopy Point of View},
{Lecture Notes in Math. \textbf{161}, 1970}

\bibitem[Sug60]{Sug60}
{M. Sugawara},
\textit{On the homotopy-commutativity of groups and loop spaces},
{Mem. College Sci. Univ. Kyoto Ser. A Math. \textbf{33} (1960), 257--269}.

\bibitem[Tsu12]{Tsu12}
{M. Tsutaya},
\textit{Finiteness of $A_n$-equivalence types of gauge groups},
{J. London Math. Soc. \textbf{85} (2012), 142--164}.

\bibitem[Tsu15]{Tsu15}
{M. Tsutaya},
\textit{Homotopy pullback of $A_n$-spaces and its applications to $A_n$-types of gauge groups},
{Topology Appl. \textbf{187} (2015), 1--25}.

\bibitem[Wil69]{Wil69}
{F. D. Williams},
\textit{Higher homotopy-commutativity},
{Trans. Amer. Math. Soc. \textbf{139} (1969), 191--206}.

\end{thebibliography}
\end{document}